\documentclass[11pt]{article}
\usepackage[utf8x]{inputenc}
\usepackage[margin=1in]{geometry}
\usepackage{amsthm}
\usepackage{amsmath, amssymb}
\newcommand{\Reals}[1]{\mathbb{R}^{#1}}
\newcommand{\bigO}[1]{\mathrm{O}(#1)}
\newcommand{\littleO}[1]{\mathrm{o}(#1)}
\newtheorem{assumption}{Assumption}

\newtheorem{theorem}{Theorem}
\newtheorem{proposition}{Proposition}
\newtheorem{lemma}{Lemma}
\newtheorem{corollary}{Corollary}
\usepackage{mathtools}
\usepackage{color}
\newcommand\defeq{\stackrel{\mathclap{\normalfont\mbox{}}}{=}}
\newcommand{\commentEq}[1]{
\hspace{10px}\text{ \footnotesize [{\em #1}]}}

\newcommand{\Ex}[1]{\mathrm{E}#1}
\newcommand{\ExCond}[2]{\Ex\left(#1|#2\right)}

\newcommand{\nn}{\nonumber}

\newcommand{\thetaBar}[1]{\overline{\theta}_{#1}}
\newcommand{\eps}{\epsilon}

\newcommand{\Fn}[1]{\mathcal{F}_{#1}}

\newcommand{\thetastar}{\theta_\star}
\newcommand{\thetarm}[1]{\theta_{#1}}
\newcommand{\thetaim}[1]{\theta_{#1}}

\newcommand{\m}[1]{\mathrm{\uppercase{#1}}}

\newcommand{\thetaMed}[1]{\theta_{#1}^+}

\usepackage{color}
\usepackage{enumerate}
\usepackage{graphicx}
\usepackage{natbib}
 \bibpunct{(}{)}{;}{a}{,}{,}
\usepackage{algorithm}
\usepackage{algorithmic}
\usepackage{multirow}

\usepackage{tikz}
\usepackage{tikzscale}
\usepackage{filecontents}    %% only for this demo

\newcommand{\StateAssumptions}{Symbol $\|\cdot\|$ denotes the $L_2$ vector/matrix norm.
The parameter space for $\theta$ is $\Theta\subseteq\Reals{p}$, and is convex.
For positive scalar sequences $(a_n)$ and $(b_n)$, we write $b_n=\bigO{a_n}$ to express that $b_n \le c a_n$, for some fixed $c>0$, and every $n=1, 2,\ldots$; we write $b_n = \littleO{a_n}$ to express that $b_n/a_n\to0$ in the limit where $n\to\infty$.
Notation $b_n\downarrow 0$ means that $b_n$ is positive and decreasing towards zero.
To ensure existence of $\thetaMed{n}$ as a solution in Equation~\ref{eq:spp0},  we will assume throughout this paper that that a convex scalar potential, $\P$, exists such that $\nabla \P=h$.
This assumption is not strictly necessary. In Section~\ref{section:example}, 
for instance, we study a quantile regression problem where $h$ is scalar-valued and non-decreasing, 
which ensures the existence of $F$ and $\theta_n^+$. 
Depending on which result we state, the stochastic proximal point 
algorithm operates under a combination of
the following assumptions.

\begin{assumption}
\label{assumption:gamma}
It holds that $\gamma_n = \gamma_1 n^{-\gamma}$,
	$\gamma_1>0$ and $\gamma \in (0, 1]$.
\end{assumption}

\begin{assumption}
\label{assumption:lip}
Function $h$ is Lipschitz with parameter $L$, i.e., for all $\theta_1, \theta_2\in\Theta$,
\begin{align}
\|h(\theta_1)-h(\theta_2)\| \le L \|\theta_1-\theta_2\|.\nn
\end{align}
\end{assumption}

\begin{assumption}
\label{assumption:convexity}
Function $h$ satisfies either
\begin{enumerate}[(a)]
\item \label{A:h_inward}
$(\theta-\thetastar)^\top h(\theta) \ge 0$, for all $\theta\in\Theta$;
\item \label{A:h_inwardstrict}
	$(\theta-\thetastar)^\top h(\theta) > 0$, for all $\theta\in\Theta\setminus\{\thetastar\}$;
 \item \label{A:h_inward2}
$(\theta-\thetastar)^\top h(\theta) \ge \mu \|\theta-\thetastar\|^2$, for some fixed	  $\mu>0$, and all $\theta\in\Theta$.
\end{enumerate}
\end{assumption}

\begin{assumption}
\label{assumption:errors}
There exists fixed $\sigma^2>0$ such that, for all $n=1, 2, \ldots$,
\begin{align}
\Ex(\varepsilon_n | \Fn{n-1}) = 0,~\text{and}~\Ex(\|\varepsilon_n\|^2 | \mathcal{F}_{n-1}) \le \sigma^2.\nn
\end{align}
\end{assumption}

\begin{assumption}
\label{assumption:Lindeberg}
	Let $\Xi_n = \ExCond{\varepsilon_n
	\varepsilon_n^\top}{\mathcal{F}_{n-1}}$, then 
	$\|\Xi_n - \Xi\| \to 0$ for fixed positive-definite matrix $\Xi$.
Furthermore, if $\sigma_{n, s}^2 = \Ex(\mathbb{I}_{\|\varepsilon_n\|^2 \ge
	s/\gamma_n}\|\varepsilon_n\|^2)$, then for all $s>0$, $\sum_{i=1}^n
	\sigma_{i,s}^2 = \littleO{n}$ if $\gamma_n \propto n^{-1}$, or $\sigma_{n,
	s}^2 = \littleO{1}$ otherwise.
\end{assumption}
}

\newcommand{\assumeMain}{}

\newcommand{\aGn}{\ref{assumption:gamma}}
\newcommand{\assumeGn}{Assumption \aGn}
\newcommand{\aLip}{\ref{assumption:lip}}
\newcommand{\assumeLip}{Assumption \aLip}
\newcommand{\aConvex}{\ref{assumption:convexity}(\ref{A:h_inward})}
\newcommand{\aConvexStrict}{\ref{assumption:convexity}(\ref{A:h_inwardstrict})}
\newcommand{\assumeConvex}{Assumption \aConvex}
\newcommand{\assumeConvexStrict}{Assumption \aConvexStrict}
\newcommand{\aStrongConvex}{\ref{assumption:convexity}(\ref{A:h_inward2})}
\newcommand{\assumeStrongConvex}{Assumption \aStrongConvex}

\newcommand{\aErrors}{\ref{assumption:errors}}
\newcommand{\assumeErrors}{Assumption \aErrors}
\newcommand{\aLind}{\ref{assumption:Lindeberg}}
\newcommand{\assumeLind}{Assumption \aLind}

\newcommand{\TheoremConvergence}{
\vspace{5px}
\begin{theorem}
\label{theorem:convergence}
Suppose that Assumptions \aGn, \aLip, \aConvexStrict, and \aErrors\  hold. 
Then, the iterates $\thetaim{n}$ of 
the stochastic proximal point algorithm of Equation~\eqref{eq:spp} converges almost surely to $\thetastar$; i.e., $\thetaim{n} \to \thetastar$,
such that $h(\thetastar)=0$, almost surely.
\end{theorem}
}

\renewcommand{\thetaBar}[1]{\thetaMed{#1}}
\newcommand{\TheoremConvergenceProof}{
\TheoremConvergence
\begin{proof}
By Equation~\eqref{eq:spp}:
\begin{align}
\label{thm2:1}
\|\thetaim{n}-\thetastar\|^2 = \|\thetaim{n-1}-\thetastar\|^2 - 2\gamma_n (\thetaim{n-1}-\thetastar)^\top 
(h(\theta_n^+) + \varepsilon_n) + \gamma_n^2 \|h(\theta_n^+) + \varepsilon_n\|^2.
\end{align}
We use $(\thetaim{n-1}-\thetastar) = (\thetaBar{n} - \thetastar) + (\thetaim{n-1}-\thetaBar{n})$, and that $\theta_{n-1} - \theta_n^+ = \gamma_n h(\theta_n^+)$
from Equation~\eqref{eq:FP} in order to obtain:
\begin{align}
\label{thm2:2}
 R_n \triangleq \ExCond{(\thetaim{n-1}-\thetastar)^\top (h(\theta_n^+) + \varepsilon_n)}{\mathcal{F}_{n-1}}
& =  (\thetaBar{n}-\thetastar)^\top h(\thetaBar{n}) + 
	(\thetaim{n-1}-\thetaBar{n})^\top h(\thetaBar{n}) \nonumber \\
& = (\thetaBar{n}-\thetastar)^\top h(\thetaBar{n}) + 
	\gamma_n \|h(\thetaBar{n})\|^2 > 0.
\commentEq{by \assumeConvex}
\end{align}
Taking norms in Equation~\eqref{eq:FP} we obtain:
\begin{align}
\label{thm2:2ba}
\|\thetaim{n-1}-\thetastar\|^2 & = \|\thetaMed{n}-\thetastar\|^2
 + 2\gamma_n h(\thetaMed{n})^\top(\thetaMed{n}-\thetastar)
+\gamma_n^2 \|h(\thetaMed{n})\|^2,\nn\\
& > \|\thetaMed{n}-\thetastar\|^2.
\commentEq{by \assumeMain\assumeConvex}
\end{align}
It follows that
\begin{align}
\label{thm2:2b}
\|h(\thetaBar{n})\|  = \|h(\thetaMed{n})-h(\thetastar)\| 
&  \le L \|\thetaMed{n} - \thetastar\| 
\commentEq{by \assumeLip}\nn\\
& \le L \|\thetaim{n-1}-\thetastar\|.
\commentEq{by Inequality~\eqref{thm2:2ba}}
\end{align}

Furthermore, 
\begin{align}
\label{thm2:2c}
 \ExCond{\|h(\theta_n^+) + \varepsilon_n\|^2}{\mathcal{F}_{n-1}} & \le 2\|h(\thetaBar{n})\|^2 + 2\ExCond{\|\varepsilon_n\|^2}{\mathcal{F}_{n-1}} \nn\\
 & \le 2L^2 \|\thetaim{n-1}-\thetastar\|^2 + 2\sigma^2.
\commentEq{by Inequality~\eqref{thm2:2b} and \assumeMain\assumeErrors}
\end{align}

Taking expectations in Equation~\eqref{thm2:1} conditional on $\Fn{n-1}$, and using Equation~\eqref{thm2:2} and Inequality~\eqref{thm2:2c} we obtain
\begin{align}
\label{thm2:3}
\ExCond{\|\thetaim{n}-\thetastar\|^2}{\Fn{n-1}}
\le (1 + 2\gamma_n^2 L^2) \|\thetaim{n-1}-\thetastar\|^2 - 2\gamma_n R_n+ 2\gamma_n^2 \sigma^2.
\end{align}
We now use an argument---due to~\citet{gladyshev1965stochastic}--- 
that is also applicable to the classical Robbins-Monro procedure;
see, for example, \citet[Section 5.2.2]{benveniste1990adaptive}, 
or \citet[Theorem 1.9]{ljung1992stochastic}.
Random variable $R_n$ is positive by Inequality~\eqref{thm2:2}, 
and $\sum \gamma_i  = \infty$ and $\sum \gamma_i^2 < \infty$ by \assumeGn.
Therefore, we can invoke the supermartingale lemma of \citet{robbins1985convergence} to infer that $\|\thetaim{n}-\thetastar\|^2 \to B>0$  and  $\sum \gamma_n R_n < \infty$, almost surely.
If $B\ne0$ then $\lim \inf \|\thetaim{n}-\thetastar\| >0$, 
and thus the series $\sum_n \gamma_n R_n$ diverges by Inequality~\eqref{thm2:2}
and $\sum \gamma_i=\infty$ (\assumeGn). This is a contradiction. Thus, $B=0$. 
\end{proof}
}

\newcommand{\TheoremDeviance}{
\vspace{5px}
\begin{theorem}
\label{theorem:deviance}
Suppose that Assumptions~\aGn, \aLip, \aConvex, and \aErrors\ hold. Let $\Gamma^2 = \Ex{\|\thetaim{0}-\thetastar\|^2} + \sigma^2 \sum_{i=1}^\infty \gamma_i^2 + \gamma_1^2 \sigma^2$.
Then, if $\gamma \in (2/3, 1]$, there exists 
$n_{0,1} < \infty$ such that, for all $n > n_{0,1}$, the 
iterate $\theta_n$ of the stochastic proximal point algorithm of Equation~\eqref{eq:spp} satisfies:
\begin{align}
\Ex(\P(\thetaim{n}) - \P(\thetastar)) \le \left[\frac{2\Gamma^2}{\gamma \gamma_1}+\littleO{1}\right] n^{-1+\gamma}.\nn
\end{align}
If $\gamma \in (1/2, 2/3)$, there exists 
$n_{0,2} < \infty$ such that, for all $n > n_{0,2}$, 
\begin{align}
\Ex(\P(\thetaim{n}) - \P(\thetastar)) \le \left[\Gamma \sigma \sqrt{L\gamma_1} +\littleO{1}\right]n^{-\gamma/2}.\nn
\end{align}
Otherwise, $\gamma=2/3$ and  there exists $n_{0,3} < \infty$ such that, for all $n > n_{0,3}$, 
\begin{align}
\Ex(\P(\thetaim{n}) - \P(\thetastar)) \le \left[\frac{3 + \sqrt{9 + 4\gamma_1^3L\sigma^2/\Gamma^2}}{2\gamma_1/\Gamma^2} +\littleO{1}\right]n^{-1/3}.\nn
\end{align}
\end{theorem}
}
\newcommand{\hMed}[1]{h_{#1}^{+}}
\newcommand{\hMedsq}[1]{{\hMed{#1}}^2}
\newcommand{\bias}[1]{\|\thetaim{#1}-\thetastar\|^2}

\newcommand{\TheoremDevianceProof}{
\TheoremDeviance
\begin{proof}
Note that
$\thetaBar{n} + \gamma_n h(\thetaBar{n}) = \thetaim{n-1}$ is equivalent to
$\thetaBar{n} = \arg \min_{\theta} \{ \frac{1}{2\gamma_n} \|\theta-\thetaim{n-1}\|^2 + \P(\theta)\}$.
Therefore, comparing the values of the expression 
for $\theta = \thetaBar{n}$ and $\theta = \thetaim{n-1}$, we obtain
\begin{align}
\label{thm4:1}
\P(\thetaBar{n}) + \frac{1}{2\gamma_n} \|\thetaBar{n} - \thetaim{n-1}\|^2
\le \P(\thetaim{n-1}).
\end{align}
Since $\thetaim{n-1} - \thetaBar{n} = \gamma_n h(\thetaBar{n})$, Inequality~\eqref{thm4:1} 
can be written as 
\begin{align}
\label{thm4:2}
\P(\thetaim{n-1}) - \P(\thetaBar{n}) - \frac{1}{2} \gamma_n \|h(\thetaBar{n})\|^2 \ge 0.
\end{align}
Note that $\P(\theta_\star) \le \P(\theta)$, for all $\theta$. Thus, we have:
\begin{align}
\label{thm4:3}
\P(\thetaBar{n}) - \P(\thetastar) & \le h(\thetaBar{n})^\top (\thetaBar{n}-\thetastar) \commentEq{by convexity~\assumeConvex}\nn\\
\P(\thetaBar{n}) - \P(\thetastar) & \le \|h(\thetaBar{n})\| \cdot \|\thetaBar{n}-\thetastar\|\nn\\
[\Ex(\P(\thetaBar{n}) - \P(\thetastar))]^2 & \le [\Ex(\|h(\theta_n^+)\| \cdot \|(\theta_n^+ - \theta_\star\|)]^2\nn\\
[\Ex(\P(\thetaBar{n}) - \P(\thetastar))]^2 & \le \Ex(\|h(\thetaBar{n})\|^2) \Ex(\|\thetaBar{n} - \thetastar\|^2)
\commentEq{by Cauchy-Schwarz inequality}.
\end{align}
Therefore,
\begin{align}
\label{thm4:4}
\Ex(\bias{n}) & = \Ex(\|\thetaBar{n}-\thetastar\|^2) - 2\gamma_n \Ex((\thetaMed{n}-\thetastar)^\top \varepsilon_n) + \gamma_n^2 \Ex(\|\varepsilon_n\|^2)\nn\\
&=  \Ex(\|\thetaBar{n}-\thetastar\|^2) + \gamma_n^2 \Ex(\|\varepsilon_n\|^2)
 \nn\\
&\le \Ex(\bias{n-1}) + \gamma_n^2 \sigma^2.
\commentEq{by Inequality~\eqref{thm2:2ba} and \assumeMain\assumeErrors} \nn\\
& \le \Ex(\bias{0}) + \sigma^2 \sum_{i=1}^n \gamma_i^2.
\commentEq{by induction.}
\end{align}
For brevity, 
define $h_n =\Ex(\P(\thetaim{n})-\P(\thetastar))$ and $\hMed{n} = \Ex(\P(\thetaBar{n})-\P(\thetastar))$. It follows that $h_n >0, h_n^+ > 0$, everywhere. We want to derive a bound for $h_n$.
Since $\Ex(\varepsilon_n | \Fn{n-1}) = 0$, it follows from \assumeMain\assumeErrors\
that $\Ex(\|\thetaBar{n}-\thetastar\|^2) \le \Ex(\bias{n})  + \gamma_n^2 \sigma^2$. 
Using Inequality~\eqref{thm4:4}, we get
\begin{align}
\label{thm4:4b}
\Ex(\|\thetaBar{n}- \thetastar\|^2) & \le \Ex(\bias{0}) + \sigma^2 \sum_{i=1}^\infty \gamma_i^2 + \gamma_n^2 \sigma^2  \le \Gamma^2.
\end{align}
From Inequality~\eqref{thm4:3} and Inequality~\eqref{thm4:4b}, we get
\begin{align}
\label{thm4:6}
\Ex(\|h(\thetaBar{n})\|^2) \ge \frac{1}{\Gamma^2} [\Ex(\P(\thetaBar{n}) - \P(\thetastar))]^2 = \frac{1}{\Gamma^2} {\hMed{n}}^2.
\end{align}
Furthermore, by convexity of $F$,~\assumeMain\assumeConvex, and \assumeMain\assumeErrors, we have that
\begin{align}
\label{thm4:5}
\P(\thetaim{n})  & = \P(\thetaMed{n} - \gamma_n\varepsilon_n) \nn\\
\P(\thetaim{n}) & \le \P(\thetaBar{n}) - \gamma_n h(\thetaMed{n})^\top \varepsilon_n + \gamma_n^2 \frac{L}{2} \|\varepsilon_n\|^2 
\commentEq{by Lipschitz continuity}\nn\\
\P(\thetaim{n}) - \P(\thetastar)& \le \P(\thetaBar{n}) - \P(\thetastar)- \gamma_n h(\thetaMed{n})^\top \varepsilon_n + \gamma_n^2 \frac{L}{2} \|\varepsilon_n\|^2 
\nn\\
h_n & \le \hMed{n} + \gamma_n^2 \frac{L\sigma^2}{2}.
\commentEq{by taking expectations.}
\end{align}
Now, in Inequality \eqref{thm4:2}, we substract $\P(\thetastar)$ from the left-hand side, 
take expectations, and combine with Inequality~\eqref{thm4:6} to obtain
\begin{align}
\label{thm4:7}
h_{n-1} \ge \hMed{n} + \frac{1}{2 \Gamma^2} \gamma_n \hMedsq{n} \triangleq R_{\gamma_n}(\hMed
{n}).
\end{align}
Function $R_{\gamma_n}(x)$ defines a nondecreasing map, since its argument, $h_n^+$, is always positive. Let $R_{\gamma_n}^{-1}$ denote its inverse, which is also nondecreasing. Thus, we obtain $\hMed{n} \le R_{\gamma_n}^{-1}(h_{n-1})$. 
Using Equation \eqref{thm4:7}, we can rewrite Inequality~\eqref{thm4:5} as
\begin{align}
\label{thm4:8}
h_n \le R_{\gamma_n}^{-1}(h_{n-1}) +\gamma_n^2 \frac{L \sigma^2}{2}.
\end{align}

Inequality~\eqref{thm4:8} is our main recursion, since ultimately we want to upper-bound $h_n$. Our solution strategy is as follows.
We will try to find a base sequence $(b_n)$ such that 
$b_n \ge R_{\gamma_n}^{-1}(b_{n-1}) + \gamma_n^2 \frac{L \sigma^2}{2}$.
Since one can take $b_n$ to be increasing arbitrarily, we will try to 
find the smallest possible sequence $(b_n)$ that satisfies the recursion. To make our analysis more tractable we will 
search in the family of sequences $b_n = b_1 n^{-\beta}$, for various values 
$b_1, \beta>0$. Then, $b_n$ will be an upper-bound for $h_n$. To see this inductively, 
assume that $h_{n-1} \le b_{n-1}$ and that $h_n$ satisfies \eqref{thm4:8}.
Then, 
$h_n \le R_{\gamma_n}^{-1}(h_{n-1}) +\gamma_n^2 \frac{L \sigma^2}{2}
\le R_{\gamma_n}^{-1}(b_{n-1}) + \gamma_n^2 \frac{L \sigma^2}{2} \le b_n$, 
where the first inequality follows from the monotonicity of $R_{\gamma_n}$, 
and the second inequality follows from definition of $b_n$.

Now, the condition for $b_n$ can be rewritten as 
$b_{n-1} \le R_{\gamma_n}(b_n - \gamma_n^2 \frac{L \sigma^2}{2})$, and 
by definition of $R_{\gamma_n}$ we get
\begin{align}
\label{thm4:9}
b_{n-1} \le b_n - \gamma_n^2 \frac{L \sigma^2}{2} +\gamma_n \frac{1}{2\Gamma^2} (b_n - \gamma_n^2 \frac{L\sigma^2}{2})^{2}
\end{align}
Using $b_n = b_1 n^{-\beta}$ and $\gamma_n = \gamma_1 n^{-\gamma}$ (\assumeMain\assumeGn), we obtain
\begin{align}
\label{thm4:10}
b_1[(n-1)^{-\beta} - n^{-\beta}]  
+ \frac{L \sigma^2 \gamma_1^2}{2} n^{-2\gamma} 
+ \frac{L \sigma^2\gamma_1^3 b_1}{2\Gamma^2} n^{-\beta-3\gamma}
 - \frac{\gamma_1b_1^2}{2\Gamma^2} n^{- 2\beta-\gamma}
- \frac{L^2\sigma^4\gamma_1^5}{8\Gamma^2} n^{-5\gamma} \le 0.
\end{align}

We have $(n-1)^{-\beta} - n^{-\beta} < \frac{1}{1-\beta} n^{-1-\beta}$, 
for $n>1$. Thus, it suffices to have
\begin{align}
\label{thm4:11}
\frac{b_1}{1-\beta} n^{-1-\beta}
+ \frac{L \sigma^2 \gamma_1^2}{2} n^{-2\gamma} 
+ \frac{L \sigma^2\gamma_1^3 b_1}{2\Gamma^2} n^{-\beta-3\gamma}
 - \frac{\gamma_1 b_1^2}{2\Gamma^2} n^{-2\beta-\gamma}\le 0,
\end{align}
where we dropped the $n^{-5\gamma}$ term without loss of generality. 
The positive terms in Inequality \eqref{thm4:11} are $n^{-1-\beta}, n^{-2\gamma},$ and $n^{-\beta-3\gamma}$, and the only negative term is of order $n^{-2\beta-\gamma}$.
In order to find the largest possible $\beta$ to satisfy \eqref{thm4:11}, one needs to
equate the term $n^{-2\beta -\gamma}$ with the slowest possible term with a positive
coefficient, i.e., set $2\beta + \gamma = \min\{1+\beta, \beta + 3\gamma, 2\gamma\}$.
However, $\beta + 3\gamma > 1+\beta$ and $\beta + 3\gamma > 2\gamma$, 
and thus $2\beta+\gamma=\min\{1+\beta, 2\gamma\}$, which 
implies only three cases:
\begin{enumerate}[(a)]
\item $1+\beta < 2\gamma$, and thus 
$ 2\beta + \gamma = 1+\beta$, which implies $\beta = 1-\gamma$.
	Also, $1 + \beta  < 2\gamma \Rightarrow 2-\gamma  < 2 \gamma$, 
	and thus $\gamma  \in (2/3, 1]$. In this case, $b_1$ will satisfy \eqref{thm4:11} for all $n >n_{0,1}$, for some $n_{0,1}$, if
	\begin{align}
	\label{thm4:10}
	\frac{b_1}{1-\beta} < \frac{\gamma_1b_1^2}{2\Gamma^2} \Leftrightarrow
	b_1 >  \frac{2\Gamma^2}{\gamma \gamma_1}.
\end{align}

\item $2\gamma < 1+\beta$, and thus 
$ 2\beta + \gamma = 2\gamma $, which implies $\beta = \gamma/2$.
	Also, $1 + \beta > 2\gamma \Rightarrow 1+\gamma/2 > 2 \gamma$, 
	and thus $\gamma  \in (1/2, 2/3)$. In this case, $b_1$ 
	will satisfy \eqref{thm4:11} for all $n >n_{0,2}$, for some $n_{0,2}$, if
	\begin{align}
	\label{thm4:12}
	\frac{\gamma_1^2 L \sigma^2}{2} < \frac{\gamma_1b_1^2}{2\Gamma^2} \Leftrightarrow
	b_1 >  \Gamma \sigma \sqrt{L \gamma_1}.
	\end{align}

\item $2\gamma=1+\beta$, and thus $2\gamma= 1+\beta=2\beta + \gamma$, 
which solves to $\gamma=2/3$ and $\beta=1/3$. 
In this case, we need
\begin{align}
	\label{thm4:10a}
	\frac{b_1}{1-\beta} + \frac{\gamma_1^2 L \sigma^2}{2} < \frac{\gamma_1b_1^2}{2\Gamma^2}.
\end{align}
Because all constants are positive in Inequality~\eqref{thm4:10a}, including $b_1$, it follows that
\begin{align}
\label{thm4:10b}
b_1 > \frac{3 + \sqrt{9 + 4\gamma_1^3L\sigma^2/\Gamma^2}}{2\gamma_1/\Gamma^2}.
\end{align}

\end{enumerate}
\noindent{\em Remarks.} The constants $n_{0,1}, n_{0,2}, n_{0,3}$ depend on the problem parameters and the desired accuracy in the bounds of Theorem \ref{theorem:deviance}.
It is straightforward to derive exact values for them.
For example, consider case $(a)$ and assume we picked $b_1$ such that 
$ \frac{\gamma_1b_1^2}{2\Gamma^2} - \frac{b_1}{1-\beta}  = \epsilon>0$.
Ignoring the term $n^{-3\gamma-\beta}$ (for simplicity), 
Inequality~\eqref{thm4:11} becomes 
\begin{align}
\epsilon n^{-2+\gamma} \ge \frac{L\sigma^2\gamma_1^2}{2} n^{-2\gamma}
\Rightarrow n \ge (\frac{L\sigma^2\gamma_1^2}{2\epsilon})^c \equiv n_{0,1},
\end{align}
where $c=1/(3\gamma-2) >0$ since $\gamma \in (2/3, 1]$.
Parameter $n_{0,1}$ can therefore be set according to desired accuracy $\epsilon$.
Similarly, we can derive expressions for $n_{0,2}$ and $n_{0,3}$. 
\end{proof}
}

\newcommand{\TheoremMSE}{
\vspace{5px}
\begin{theorem}
\label{theorem:nonasymptotic}
Suppose that Assumptions~\aGn, \aStrongConvex, and \aErrors\ hold.
Let $\zeta_n = \Ex(\|\thetaim{n}-\thetastar\|^2)$ and define $\kappa= 1+2\gamma_1\mu$, 
where the $\theta_n$ is the $n$-th iterate of the stochastic proximal point algorithm of Equation~\eqref{eq:spp}.
Then, if $\gamma <1$, for every $n > 1$ it holds that
\begin{align}
\zeta_n \le \exp\{-\log \kappa \cdot n^{1-\gamma}\}  \zeta_0 + \sigma^2 \frac{\gamma_1 \kappa}{\mu} n^{-\gamma} + \bigO{n^{-\gamma-1}}.\nn
\end{align}
Otherwise, if $\gamma=1$, it holds that
\begin{align}
\zeta_n \le 
\exp\{-\log \kappa \cdot  \log n\} \zeta_0 +\sigma^2 \frac{\gamma_1 \kappa}{\mu} n^{-1} + \bigO{n^{-2}}.\nn
\end{align}
\end{theorem}
}

\newcommand{\TheoremMSEProof}{
\TheoremMSE
\begin{proof}
First we prove two lemmas that will be useful for Theorem \ref{theorem:nonasymptotic}.
\begin{lemma}
\label{lemma:decay_factor}
Consider a sequence $b_n$ such that $b_n \downarrow 0$
and $\sum_{i=1}^\infty b_i = \infty$. Then,
there exists a positive constant $K>0$, such that
\begin{align}
\label{eq:decay_factor}
\prod_{i=1}^n \frac{1}{1+b_i} \le \exp(-K  \sum_{i=1}^n b_i).
\end{align}
\end{lemma}
\begin{proof}
The function $x \log(1+1/x)$ is increasing-concave in $(0, \infty)$. 
From $b_n \downarrow 0$ it follows that $\log(1+b_n) / b_n$ is non-increasing.
Consider the value $K = \log(1 + b_1) / b_1$.
Then, $(1+b_n)^{-1} \le \exp(-K b_n)$.
Successive applications of this inequality
yields Inequality \eqref{eq:decay_factor}.
\end{proof}
\begin{lemma}[\cite{toulis2017aos}]
\label{lemma:implicit_recursion}
Consider sequences $a_n \downarrow 0, b_n \downarrow 0$, and $c_n \downarrow 0$ such that, $a_n = \littleO{b_n}$,  $\sum_{i=1}^\infty a_i = A < \infty$, and there is $n'$ such that $c_n/b_n < 1$ for all $n>n'$.
Define, 
\begin{align}
\label{new:defs}
\delta_n \triangleq \frac{1}{a_n} (a_{n-1}/b_{n-1} -a_n/{b_n}) \text{ and }
\zeta_n \triangleq \frac{c_n}{b_{n-1}} \frac{a_{n-1}}{a_n},
\end{align}
and suppose that $\delta_n \downarrow 0$ and $\zeta_n  \downarrow 0$.
Pick a positive $n_0$ such that $\delta_n + \zeta_n < 1$
and $(1+c_n)/(1+b_n) < 1$, for 
all $n \ge n_0$.\\
Consider a positive sequence $y_n>0$
that satisfies the recursive inequality,
\begin{align}
\label{ineq:implicit_recursion}
y_n \le \frac{1+c_n}{1 + b_n} y_{n-1}  + a_n.
\end{align}
Then, for every $n>0$,
\begin{align}
\label{eq:result:implicit_recursion}
y_n \le K_0 \frac{a_n}{b_n} + Q_{1}^n y_0 + Q_{n_0+1}^n (1+c_1)^{n_0} A,
 \end{align}
 where $K_0 = (1+b_1) \left(
 1-\delta_{n_0}  - \zeta_{n_0}\right)^{-1}$, $Q_i^n = \prod_{j=i}^n (1+c_i)/(1+b_i)$,
and $Q_i^n = 1$ if $n < i$, by definition.
\end{lemma}
\begin{corollary}
\label{corollary:implicit_recursion}
In Lemma \ref{lemma:implicit_recursion}
assume $a_n = a_1 n^{-\alpha}$ and $b_n = b_1 n^{-\beta}$, and $c_n=0$, 
where $\alpha > \beta$, and $a_1, b_1, \beta>0$ and $1 < \alpha < 1+\beta$. 
Then,
\begin{align}
\label{thm1:eq4}
y_n \le 2\frac{a_1 (1+b_1)}{b_1} n^{-\alpha + \beta} + \exp(-\log(1+b_1)n^{1-\beta}) [y_0 +(1+b_1)^{n_0}A],
 \end{align}
 where $n_0>0$ and $A=\sum_i a_i < \infty$.
\end{corollary}
\begin{proof}
In this proof, we will assume, for simplicity,
$(n-1)^{-c} - n^{-c} \le n^{-1-c}$, $c \in (0, 1)$, for every $n>0$.
It is straightforward to derive an appropriate bound for each value of $c$.
Furthermore, we assume $\sum_{i=1}^n i^{-\gamma} \ge n^{1-\gamma}$, 
for every $n>0$.
Formally, this holds for $n \ge n'$, where $n'$ in practice is very small (e.g., $n'=14$ if $\gamma=0.1$, 
$n'=5$ if $\gamma=0.5$, and $n'=9$ if $\gamma=0.9$, etc.)\newline
By definition,
\begin{align}
\delta_n = \frac{1}{a_n} (\frac{a_{n-1}}{b_{n-1}}-\frac{a_n}{b_n}) & = 
\frac{1}{a_1 n^{-\alpha}} \frac{a_1}{b_1} ((n-1)^{-\alpha+\beta} - n^{-\alpha+\beta}) \nn\\
& = \frac{1}{n^{-\alpha} b_1} [(n-1)^{-\alpha+\beta} - n^{-\alpha+\beta}]\nn\\
& \le \frac{1}{b_1} n^{-1+\beta}.
\end{align}
Also, $\zeta_n = 0$ since $c_n=0$.
We can take $n_0 = \lceil (2/b_1)^{1/(1-\beta)} \rceil$,
for which $\delta_{n_0} \le 1/2$.
Therefore, 
$K_0 = (1+b_1)(1-\delta_{n_0})^{-1} \le 2 (1+b_1)$;
we can simply take $K_0 =2(1+b_1)$.
Since $c_n =0$, $Q_{i}^n = \prod_{j=i}^n (1+b_i)^{-1}$.
Thus, 
\begin{align}
\label{cor:ineqs}
Q_1^n & \ge (1+b_1)^{-n}, \text{ and } \nn \\
Q_1^n & \le \exp(-\log(1+b_1)/b_1 \sum_{i=1}^n b_i),
 \commentEq{by Lemma \ref{lemma:decay_factor}.} \nonumber \\
 Q_1^n & \le  \exp(-\log(1+b_1) n^{1-\beta}).
 \commentEq{because $\sum_{i=1}^n i^{-\beta} \ge n^{1-\beta}$.}
\end{align}
 Lemma \ref{lemma:implicit_recursion} and Ineqs. 
 \eqref{cor:ineqs} imply
\begin{align}
y_n & \le K_0 \frac{a_n}{b_n} + Q_{1}^n y_0 + Q_{n_0+1}^n (1+c_1)^{n_0} A 
\commentEq{by Lemma  \ref{lemma:implicit_recursion} }
\nn \\
 & \le 2\frac{a_1 (1+b_1)}{b_1} n^{-\alpha + \beta} + Q_1^{n} [y_0 + (1+b_1)^{n_0}A] 
 \commentEq{by Ineqs. \eqref{cor:ineqs}, $c_1=0$}
 \nn \\
&  \le  2\frac{a_1 (1+b_1)}{b_1} n^{-\alpha + \beta} + \exp(-\log(1+b_1)n^{1-\beta}) [y_0 +(1+b_1)^{n_0}A],
\end{align}
where the last inequality also follows from Ineqs. \eqref{cor:ineqs}.
\end{proof}
\noindent 
\textbf{Proof of Theorem \ref{theorem:nonasymptotic}.}
Now we are ready to prove the main theorem.
By definition, 
$\thetaim{n} = \thetaMed{n} - \gamma_n \varepsilon_n$, and thus, by 
\assumeErrors,
\begin{align}
\label{thm3:1}
\Ex(\|\thetaim{n}-\thetastar\|^2) \le \Ex(\|\thetaMed{n}-\thetastar\|^2) + \gamma_n^2 \sigma^2.
\end{align}
Also by definition we have $\gamma_n h(\thetaBar{n}) + \thetaBar{n} = \thetaim{n-1}$,
and thus
\begin{align}
\label{thm3:3}
\|\thetaim{n-1}-\thetastar\|^2 & = \|\thetaMed{n}-\thetastar\|^2 + 2\gamma_n(\thetaMed{n}-\thetastar)^\top h(\thetaMed{n}) + \gamma_n^2\|h(\thetaMed{n})\|^2.
\end{align}
Therefore,
\begin{align}
\label{thm3:4}
 \|\thetaBar{n}-\thetastar\|^2 + 2\gamma_n (\thetaBar{n}-\thetastar)^\top h(\thetaBar{n}) & \le \bias{n-1} \nn\\
 \|\thetaBar{n}-\thetastar\|^2 + 2\gamma_n \mu \|\thetaBar{n}-\thetastar\|^2 
 & \le \bias{n-1} 
 \commentEq{by \assumeMain\assumeStrongConvex} \nn\\
\|\thetaBar{n} - \thetastar\|^2 & \le \frac{1}{1 + 2\gamma_n \mu} \bias{n-1}. &
\end{align}

Combining Inequality~\eqref{thm3:1} and Inequality~\eqref{thm3:4} 
yields
\begin{align}
\label{thm3:5}
\Ex(\|\thetaim{n}-\thetastar\|^2)
& =  \Ex(\|\thetaMed{n}-\thetastar\|^2)+ \gamma_n^2 \sigma^2\nn\\
& \le \frac{1}{1 + 2\gamma_n \mu} \Ex(\|\thetaim{n-1}-\thetastar\|^2)+ \gamma_n^2\sigma^2.
\end{align}
The final result of Theorem \ref{theorem:nonasymptotic} is obtained 
through a direct application of Corollary \ref{corollary:implicit_recursion} 
on recursion \eqref{thm3:5}, by setting $y_n \equiv \Ex{\|\thetaim{n}-\thetastar\|^2}$, 
$b_n\equiv 2\gamma_n\mu$, and $a_n\equiv \gamma_n^2\sigma^2$.
The case where $\gamma=1$ only changes Inequality~\eqref{cor:ineqs} by replacing $\sum b_i$ with $\log n$.
\end{proof}
}

\newcommand{\TheoremNormality}{
\vspace{5px}
\begin{theorem}
\label{theorem:normality}
Suppose that Assumptions \aGn,\aLip, \aConvex, \aErrors, and \aLind\ hold, and that $(2\gamma_1J_h(\thetastar)-I)$ is positive-definite, where 
$J_h(\theta)$ is the Jacobian of $h$ at $\theta$, and $I$ is the $p\times p$ identity matrix.
Then, $\thetaim{n}$ of the stochastic proximal point algorithm of Equation~\eqref{eq:spp}
is asymptotically normal:
\begin{align}
n^{\gamma/2} (\thetaim{n}-\thetastar) \to \mathcal{N}_p(0, \Sigma).\nonumber
\end{align}
The covariance matrix $\Sigma$ is the unique solution of 
\begin{align}
(\gamma_1 \m{J}_h(\thetastar) - I/2)  \Sigma + \Sigma (\gamma_1\m{J}_h(\thetastar)-I/2)= \Xi.\nn
\end{align}
A closed-form solution for $\Sigma$ is possible if $\Xi$ commutes with $J_h(\thetastar)$, such that $\Xi J_h(\theta_\star) = J_h(\theta_\star) \Xi$.  Then, $\Sigma$ can be derived as
$\Sigma = (2\gamma_1J_h(\thetastar)-I)^{-1} \Xi$.
\end{theorem}
}
\newcommand{\TheoremNormalityProof}{
\TheoremNormality
\begin{proof}
Convergence of $\thetaim{n} \to \thetastar$ is established from 
Theorem \ref{theorem:convergence}.
By definition of the stochastic proximal point algorithm in Equation~\eqref{eq:spp},
\begin{align}
\label{thm5:1}
& \thetaim{n} = \thetaim{n-1} - \gamma_n (h(\thetaBar{n}) + \varepsilon_n), \text{ and }\\
\label{thm5:2}
& \thetaBar{n} + \gamma_n h(\thetaBar{n}) = \thetaim{n-1}.
\end{align}
We use Equation \eqref{thm5:2} and expand $h(\cdot)$ to obtain
\begin{align}
\label{thm5:3}
h(\thetaBar{n}) & = h(\thetaim{n-1}) - \gamma_n J_h(\thetaim{n-1}) h(\thetaBar{n}) + \epsilon_n \nn\\
h(\thetaBar{n}) & = \left( I + \gamma_n J_h(\thetaim{n-1})\right)^{-1} h(\thetaim{n-1})
+  \left( I + \gamma_n J_h(\thetaim{n-1})\right)^{-1} \epsilon_n,
\end{align}
where $\|\epsilon_n\| = \bigO{\gamma_n^2}$ by Theorem \ref{theorem:nonasymptotic}.
By Lipschitz continuity of $h(\cdot)$ (\assumeMain\assumeConvex)
and the almost sure convergence of $\thetaim{n}$ to $\thetastar$, 
it follows $h(\thetaim{n-1}) = J_h(\thetastar) (\thetaim{n-1}-\thetastar) + \littleO{1}$,
where $\littleO{1}$ is a vector with vanishing norm.
Therefore we can rewrite \eqref{thm5:3} as follows,
\begin{align}
\label{thm5:4}
h(\thetaBar{n}) = A_n (\thetaim{n-1}-\thetastar) + \bigO{\gamma_n^2},
\end{align}
such that $\|A_n - J_h(\thetastar)\| \to 0$, and $\bigO{\gamma_n^2}$ denotes 
a vector with norm $\bigO{\gamma_n^2}$. Thus, we can rewrite \eqref{thm5:1} as
\begin{align}
\label{thm5:5}
\thetaim{n}-\thetastar = (I - \gamma_n A_n) (\thetaim{n-1} - \thetastar) 
-\gamma_n \varepsilon_n + \bigO{\gamma_n^2}.
\end{align}
The conditions for Fabian's theorem \citep[Theorem 1]{fabian1968asymptotic}
are now satisfied, and so $\thetaim{n}-\thetastar$ is asymptotically normal with mean zero, and variance that is given in the statement of Theorem 1 by 
\citet{fabian1968asymptotic}.
\end{proof}
}

\newcommand{\LemmaChi}{
\begin{lemma}
\label{lemma:chi}
Suppose that Assumptions~\ref{assumption:lip} and \aStrongConvex{} hold and consider 
	$(x, y)\in\mathbb{R}_p^2$, two $p$-component vectors. Then, for all $n=1, 2, \ldots$:
	\begin{enumerate}[(a)]
	\item $\chi_n$ is a contraction: $ \|\chi_n(x) - \chi_n(y) \| \leq
			\frac{1}{1+\gamma_n \mu}\|x-y\|$.

		\item $ \|\chi_n(x) - x \| \leq \frac{\gamma_n L}{1+\gamma_n\mu}
			\|x-\thetastar\|$.
	\end{enumerate}
\end{lemma}
}
\newcommand{\LemmaChiProof}{
\LemmaChi
\begin{proof}
	First note that since $h(\thetastar) = 0$, $\thetastar$ is a fixed point
	of $\chi_n$.
	\begin{enumerate}[(a)]
		\item By definition of $\chi_n$ in Equation~\eqref{eq:chin}, one can write:
			\begin{displaymath}
				\chi_n(x) - \chi_n(y) = x - y +\gamma_n
				\left[h\big(\chi_n(y)\big)-h\big(\chi_n(x)\big)\right].
			\end{displaymath}
			Taking the inner product with $(\chi_n(x)-\chi_n(y))$:
			\begin{equation}
				\label{eq:foo}
			\begin{split}
				\|\chi_n(x) - \chi_n(y)\|^2
				&= (x - y)^{\top}\big(\chi_n(x)-\chi_n(y)\big)\\
				&\quad -\gamma_n
				\left[h\big(\chi_n(x)\big)-h\big(\chi_n(y)\big)\right]^{\top}
				\big(\chi_n(x)-\chi_n(y)\big).
			\end{split}
			\end{equation}
			Using \aStrongConvex, we obtain:
			\begin{displaymath}
				(1+\gamma_n\mu)\|\chi_n(x) - \chi_n(y)\|^2
				\leq (x - y)^{\top}\big(\chi_n(x)-\chi_n(y)\big),
			\end{displaymath}
			and we conclude by applying the Cauchy-Schwarz inequality to the
			right-hand side.

		\item We can write $\|\chi_n(x) - x\| = \gamma_n
			\|h\big(\chi_n(x)\big)\|$ by definition of $\chi_n$. Because
			$h\big(\chi_n(\thetastar)\big)= 0$:
			\begin{align*}
				\|\chi_n(x) - x\|
				&= \gamma_n \|h\big(\chi_n(x)\big) - h\big(\chi_n(\thetastar)\big)\|\\
				&\leq \gamma_n L\|\chi_n(x) - \chi_n(\thetastar)\|
				\leq \frac{\gamma_n L}{1+\gamma_n\mu}\|x-\thetastar\|,
			\end{align*}
			where the first inequality uses Assumption~\ref{assumption:lip} and
			the second follows from \emph{(a)}.
	\end{enumerate}
\end{proof}
}

\newcommand{\LemmaXi}{\begin{lemma} \label{lemma:xi}
		Suppose that Assumptions~\ref{assumption:lip}, \ref{assumption:errors} and \aConvex{} hold. Consider the choice of parameter $a_k = a_n,\; 1\leq k\leq K$ in
	\eqref{eq:sfp} with $a_n\leq\frac{1}{(1+\gamma_n L)^2}$, then:
\begin{displaymath}
	\ExCond{\|\thetaim{n} - \thetaMed{n}\|^2}{\Fn{n-1}}
	\leq (1-a_n)^{K}\|\thetaim{n-1} - \thetaMed{n}\|^2 + \sigma^2\gamma_n^2 a_n.
\end{displaymath}
\end{lemma} }
\newcommand{\LemmaXiProof}{\LemmaXi\begin{proof}
	Let us write $H(w_k^n, \xi_{k+1}^n) = h(w_k^n)+\varepsilon_{k+1}^n$ and define $g(x) = \gamma_n h(x) + x - \theta_{n-1}$. We can write:
	\begin{align*}
		\|w_{k+1}^n-\thetaMed{n}\|^2
	&=\|w_k^n - a_n \big(g(w_k^n) + \gamma_n\varepsilon_{k+1}^n\big) - \thetaMed{n}\|^2\\
	&= \|w_k^n-\thetaMed{n}\|^2 -2 a_n\big( g(w_k^n)
	+ \gamma_n\varepsilon_{k+1}^n\big)^T\big(w_k^n-\thetaMed{n}\big)\\
	& \quad + a_n^2\big(\|g(w_k^n)\|^2+\gamma_n^2\|\varepsilon_{k+1}^n\|^2
	+ 2 g(w_k^n)^T\gamma_n\varepsilon_{k+1}^n\big).
\end{align*}
Taking expectations on both sides conditioned on $\mathcal{F}_{n,k}$ and noting
that $\Ex(\varepsilon_{k+1}|\mathcal{F}_{n,k})=0$ and
$\Ex(\|\varepsilon_{k+1}\|^2|\mathcal{F}_{n,k})\leq\sigma^2$ by \assumeErrors{} we
get:
\begin{align*}
	\Ex(\|w_{k+1}^n-\thetaMed{n}\|^2|\mathcal{F}_{n,k})&\leq
	\|w_k^n-\thetaMed{n}\|^2 -2 a_n g(w_k^n)^T\big(w_k^n-\thetaMed{n}\big)
	+ a_n^2\|g(w_k^n)\|^2 + a_n^2\gamma_n^2\sigma^2,
\end{align*}
It follows easily from Assumptions~\ref{assumption:lip} and \aConvex{} that $g$
is $(\gamma_n L + 1)$-Lipschitz continuous and that $\big(g(x)-g(y)\big)^\top (x-y)\geq \|x-y\|^2$ for al $x$ and $y$ in $\mathbb{R}^p$. Furthermore, since $g(\thetaMed{n})=0$ by definition:
\begin{displaymath}
	\delta_{k+1}^n
	\leq \left[ 1 - 2a_n + a_n^2(1+\gamma_nL)^2\right]\delta_k
	+ a_n^2\gamma_n^2\sigma^2\;.
\end{displaymath}
where we took expectations on both sides conditioned on $\mathcal{F}_{n-1}$ and
write $\delta_k= \ExCond{\|w_k^n-\thetaMed{n}\|^2}{\mathcal{F}_{n-1}}$.
	For $a_n\leq\frac{1}{(1+\gamma_n L)^2}$, the above recursion
	becomes:
\begin{displaymath}
	\delta_{k+1}^n
	\leq (1 - a_n)\delta_k
	+ a_n^2\gamma_n^2\sigma^2\;.
\end{displaymath}
	Note that $w_K^n
	= \thetaim{n}$, and $w_1^n = \theta_{n-1}$ by definition. Therefore, we obtain:
\begin{displaymath}
	\ExCond{\|\thetaim{n} - \thetaMed{n}\|^2}{\Fn{n-1}}
	\leq (1-a_n)^{K}\|\thetaim{n-1} - \thetaMed{n}\|^2 + \sigma^2\gamma_n^2a_n\big(1-
	(1-a_n)^K\big).
\end{displaymath}
\end{proof}
}

\newcommand{\TheoremNested}{
\begin{theorem}
	\label{thm:nested}
	Suppose that Assumptions~\ref{assumption:lip}, \ref{assumption:errors} and \ref{assumption:convexity}(\ref{A:h_inward2}) hold, then
the proximal stochastic fixed point procedure in
Equation~\eqref{eq:sfp} with parameters $\gamma_n=\gamma$ and $a_k
= 2a / K$, such that $e^{-a}< \mu/L$ and $K\geq 2a (1+\gamma L)^2$,
	satisfies:
	\begin{displaymath}
	\Ex\|\thetaim{n} - \thetastar\|
		\leq C^n\|\theta_0-\thetastar\|
 + \frac{\gamma\sigma\sqrt{2a}}{(1-C)\sqrt{K}} 
			\end{displaymath}
			where $C\defeq (1+e^{-a}\gamma L)/(1+\gamma\mu)$.
\end{theorem}
}
\newcommand{\TheoremNestedProof}{
\TheoremNested

\begin{proof}
We decompose the distance between $\theta_n$ and $\theta_\star$ as the distance
between $\theta_n$ and $\thetaMed{n}$, and the distance of $\theta_n^+$ to
$\thetastar$:
\begin{align*}
	\Ex\|\thetaim{n}  - \thetastar\|
	& \leq \Ex{\|\thetaim{n} - \thetaMed{n}\|}
	+ \Ex{\|\thetaMed{n} - \thetastar\|}
	\commentEq{triangle inequality}\\
	&= \Ex{\|\thetaim{n} - \thetaMed{n}\|}
	+ \Ex{\|\chi_n(\thetaim{n-1}) - \chi_n(\thetastar)\|}
	\commentEq{by definition of $\chi_n$ in Equation~\eqref{eq:chin}}\\
	&\leq \Ex{\|\thetaim{n} - \thetaMed{n}\|}
	+ \frac{1}{1+\gamma\mu}\Ex{\|\thetaim{n-1} - \theta_{\star}\|}
	\commentEq{by Lemma~\ref{lemma:chi}~\emph{(a)}}\\
	&\leq (1-a_n)^{K/2}\Ex{\|\thetaim{n-1} - \chi_n(\thetaim{n-1})\|}
	+\sigma\gamma\sqrt{a_n}
	+ \frac{1}{1+\gamma\mu}\Ex{\|\thetaim{n-1} - \thetastar\|}
	\commentEq{by	Lemma~\ref{lemma:xi}}\\
	&\leq \frac{(1-a_n)^{K/2}\gamma L}{1+\gamma\mu}\Ex{\|\thetaim{n-1} - \thetastar\|}
	+\sigma\gamma\sqrt{a_n}
	+ \frac{1}{1+\gamma\mu}\Ex{\|\thetaim{n-1} - \thetastar\|}
	\commentEq{by Lemma~\ref{lemma:chi}\emph{(b)}}\\
	&= \left(\frac{1+(1-a_n)^{K/2}\gamma
	L}{1+\gamma\mu}\right)\Ex{\|\thetaim{n-1} - \theta_{n-1}'\|}
	+\sigma\gamma\sqrt{a_n}\,.
\end{align*}
	We now choose $a_n$ constant of the form $\frac{2a}{K}$ and obtain the following
	recursion:
\begin{align*}
	\Ex{\|\thetaim{n} - \thetastar\|}
	&\leq C\cdot \Ex{\|\thetaim{n-1} - \thetastar\|}
	+\sigma\gamma\frac{\sqrt{2a}}{\sqrt{K}},
\end{align*}
where $C$ is as in the theorem statement. Observe that for our choice of parameter, $C<1$. This recursion solves to:
	\begin{displaymath}
	\Ex{\|\thetaim{n} - \thetastar\|}
		\leq \frac{\gamma\sigma\sqrt{2a}}{(1-C)\sqrt{K}} + C^n\|\theta_0-\thetastar\|
		\,.
	\end{displaymath}
	\qed
\end{proof}

For completeness, we finally present a variant of the previous procedure, also
providing an approximate implementation of the proximal Robbins--Monro procedure
via proximal stochastic fixed points. Compared to the procedure \eqref{eq:sfp}
analyzed in Theorem~\ref{thm:nested}, we now perform an extra gradient step to
compute $\thetaim{n}$ from $\thetaim{n-1}$ instead of simply using $w_K^n$.
Formally:
\begin{equation}
	\label{eq:sfp2}
	\begin{split}
		w_1^n &= \thetaim{n-1},\\
		w_{k}^n &= w_{k-1}^n - a_k \big(\gamma_n H(w_{k-1}^n, \xi_k^n) + w_{k-1}^n
		- w_1^n\big),\quad 1 < k\leq K,\\
			\thetaim{n} &= \thetaim{n-1} -\gamma_n H(w_K^n,\xi_{K+1}^n)
	\end{split}
	\end{equation}

\begin{theorem}
	Suppose that Assumptions~\ref{assumption:lip}, \ref{assumption:errors} and \ref{assumption:convexity}(\ref{A:h_inward2}) hold, then
the procedure in
Equation~\eqref{eq:sfp2} with parameters $\gamma_n=\gamma_1/n$ and $a_k
= 2a / K$, where $a$ and $K$ are constants satisfying:
	\begin{displaymath}
		e^{-a} \leq \frac{\mu}{2\gamma_1L^2},\quad
		K\geq 3a\cdot\max\big\{(1+\gamma_1L)^2, (\gamma_1 L)^2 + e^{3a}\big\}
		\,.
	\end{displaymath}

	Then:
	\begin{displaymath}
		\Ex{\|\thetaim{n}-\thetastar\|^2}
		\leq
		\frac{e^{4\gamma_1^2\mu^2}}{n^{\gamma_1\mu}}\|\thetaim{0}-\thetastar\|^2
		+2\gamma_1^2\sigma^2e^{2\gamma_1^2\mu^2}e^{\gamma_1\mu} \cdot S(n)
		\,,
	\end{displaymath}
	where:
	\begin{displaymath}
		S(n)\leq\begin{cases}
			\frac{1}{\gamma_1\mu-1}\frac{1}{n}&\text{if $\gamma_1\mu>1$}\\
			\log(en)/n&\text{if $\gamma_1\mu=1$}\\
			\frac{2}{1-\gamma_1\mu}\frac{1}{n^{\gamma_1\mu}}&\text{if
			$\gamma_1\mu<1$}\\
		\end{cases}
	\end{displaymath}
\end{theorem}

\begin{proof}
	We focus on a single iteration $n$ and write $H(w_K^n,\xi_{K+1}^n)
	= h(w_K^n) + \epsilon_n$. We first decompose the error as usual:
	\begin{align*}
		\|\thetaim{n}-\thetastar\|^2
		&= \|\thetaim{n-1}-\gamma_n h(w_K^n)-\gamma_n\eps_n-\thetastar\|^2\\
		&= \|\thetaim{n-1}-\gamma_n h(w_K^n)-\thetastar\|^2+\gamma_n^2\|\eps_n\|^2
		-2\gamma_n \eps_n^T\big(\thetaim{n-1} - \gamma_n h(w_K^n) - \thetastar\big)
		\,.
	\end{align*}

	Recall that $\ExCond{\eps_n}{\Fn{n,K}} = 0$ and
	$\ExCond{\|\eps_n\|^2}{\Fn{n,K}}\leq\sigma^2$ by
	Assumption~\ref{assumption:errors}. Hence:
	\begin{align*}
		\ExCond{\|\thetaim{n}-\thetastar\|^2}{\Fn{n,K}}
		&\leq \|\thetaim{n-1}-\gamma_n h(w_K^n)-\thetastar\|^2+\gamma_n^2\sigma^2\\
		&= \|\thetaMed{n}+\gamma_n \big(h(\thetaMed{n})-h(w_K^n)\big)-\thetastar\|^2
		+\gamma_n^2\sigma^2
	\end{align*}
	where the equality uses that $\thetaim{n-1}-\gamma_n h(\thetaMed{n})
	= \thetaMed{n}$ by Eq.~\eqref{eq:spp0}.

	Next, using that $\|a+b\|^2 \leq (1+\alpha)\|a\|^2 + (1+\alpha^{-1})\|b\|^2$
	for all $\alpha>0$ by Young's inequality:
	\begin{align*}
		\ExCond{\|\thetaim{n}-\thetastar\|^2}{\Fn{n,K}}
		&\leq (1+\alpha)\|\thetaMed{n}-\thetastar\|^2
		+\gamma_n^2(1+\alpha^{-1})\|h(\thetaMed{n})-h(w_K^n)\|^2
		+\gamma_n^2\sigma^2\\
		&\leq \frac{1+\alpha}{(1+\gamma_n\mu)^2}\|\thetaim{n-1}-\thetastar\|^2
		+(1+\alpha^{-1})(\gamma_nL)^2\|\thetaMed{n}-w_K^n\|^2
		+\gamma_n^2\sigma^2
	\end{align*}
	where the second inequality uses Lemma~\ref{lemma:chi}~\emph{(a)} and
	Assumption~\ref{assumption:lip}.

	Taking expectations conditioned on $\Fn{n-1}$ and using
	Lemma~\ref{lemma:xi} (our choice of parameters satisfies in particular
	$a_n\leq 1/(1+\gamma_nL)^2$ as required by the Lemma):
	\begin{align*}
		\ExCond{\|\thetaim{n}-\thetastar\|^2}{\Fn{n-1}}
		&\leq \frac{1+\alpha}{(1+\gamma_n\mu)^2}\|\thetaim{n-1}-\thetastar\|^2
		+(1+\alpha^{-1})(\gamma_nL)^2(1-a_n)^K\|\thetaMed{n}-\thetaim{n-1}\|^2\\
		&\quad\quad+(1+\alpha^{-1})(\gamma_n L)^2\gamma_n^2\sigma^2a_n
		+\gamma_n^2\sigma^2\\
		&\leq
		\frac{1+\alpha+(1+\alpha^{-1})(\gamma_nL)^4(1-a_n)^K}{(1+\gamma_n\mu)^2}\|\thetaim{n-1}-\thetastar\|^2\\
		&\quad\quad+\gamma_n^2\sigma^2\big[1+(1+\alpha^{-1})(\gamma_n L)^2a_n\big]\,.
	\end{align*}
	where the second inequality uses Lemma~\ref{lemma:chi}~\emph{(b)}.

	We now pick $\alpha = (\gamma_n L)^2(1-a_n)^{K/2}$ and take expectations on
	both sides:
	\begin{displaymath}
		\Ex{\|\thetaim{n}-\thetastar\|^2}
		\leq \left(\frac{1+ (\gamma_n
		L)^2(1-a_n)^{K/2}}{1+\gamma_n\mu}\right)^2\Ex{\|\thetaim{n-1}-\thetastar\|^2}
		+\gamma_n^2\sigma^2\left[1+(\gamma_n L)^2a_n
			+ \frac{a_n}{(1-a_n)^{K/2}}\right]\,.
	\end{displaymath}
	Using the inequality $\exp\big(-nx/(1-x)\big)\leq(1-x)^n\leq\exp(-nx)$, it
	is easy to see that the choice of parameters in the theorem statement
	implies:
	\begin{displaymath}
		(1-a_n)^{K/2}\leq e^{-a},\quad
		e^{-a}(\gamma_nL)^2\leq\gamma_n\mu/2,\quad
		(\gamma_n L)^2 a_n + \frac{a_n}{(1-a_n)^{K/2}}\leq 1\;,
	\end{displaymath}
	hence the previous inequality yields:
	\begin{align*}
		\Ex{\|\thetaim{n}-\thetastar\|^2}
		&\leq \left(\frac{1+ \gamma_n\mu/2}{1+\gamma_n\mu}\right)^2\Ex{\|\thetaim{n-1}-\thetastar\|^2}
		+2\gamma_n^2\sigma^2\\
		&\leq\left(1-\frac{\gamma_n\mu}{(1+\gamma_n\mu)^2}\right)\Ex{\|\thetaim{n-1}-\thetastar\|^2}
		+ 2\gamma_n^2\sigma^2\,.
	\end{align*}

	Writing $y_n\defeq \Ex{\|\thetaim{n}-\thetastar\|^2}$, $a_n\defeq
	\gamma_n\mu/(1+\gamma_n\mu)^2$ and $b_n\defeq 2\gamma_n^2\sigma^2$,
	the previous inequality reads $y_n\leq (1-a_n)y_{n-1}+b_n$.
	Define $p_n\defeq\prod_{k=1}^n(1-a_k)$, an easy induction gives:
	\begin{equation}
		\label{eq:main_recursion}
		y_n \leq p_n y_0 + p_n\sum_{k=1}^n\frac{b_k}{p_k}
		\,.
	\end{equation}

	We first focus on getting a lower bound and upper bound on $p_n$. For the
	lower bound, using that $(1-x)\geq \exp\big(-x/(1-x)\big)$, we obtain:
	\begin{align*}
		p_n
		&\geq \exp\left(-\sum_{k=1}^n a_k\right)
		\exp\left(-\sum_{k=1}^n\frac{a_k^2}{1-a_k}\right)\\
		&\geq \exp\left(-\sum_{k=1}^n \gamma_k\mu\right)
		\exp\left(-\sum_{k=1}^n\gamma_k^2\mu^2\right)
		\geq\frac{e^{2\gamma_1^2\mu^2-\gamma_1\mu}}{n^{\gamma_1\mu}}\,.
	\end{align*}
	where the second inequality uses the definition of $a_k$ and the last
	inequality uses that $\gamma_n=\gamma_1/n$ an the series approximations of
	Lemma~\ref{lemma:series}.
	Similarly for the upper bound, using that $(1-x)\leq \exp(-x)$:
	\begin{align*}
	p_n\leq \exp\left(-\sum_{k=1}^na_k\right)
	&=\exp\left(-\sum_{k=1}^n\gamma_k\mu\right)
	\exp\left(\sum_{k=1}^n\frac{\gamma_k^2\mu^2(2+\gamma_k\mu)}{(1+\gamma_k\mu)^2}\right)\\
	&\leq\exp\left(-\sum_{k=1}^n\gamma_k\mu\right)
	\exp\left(\sum_{k=1}^n2\gamma_k^2\mu^2\right)
	\leq \frac{e^{4\gamma_1^2\mu^2}}{(n+1)^{\gamma_1\mu}}\,.
	\end{align*}

	Plugging the previous two bounds into \eqref{eq:main_recursion}, we
	obtain:
	\begin{displaymath}
		y_n\leq\frac{e^{4\gamma_1^2\mu^2}}{n^{\gamma_1\mu}}y_0
	+\frac{2\gamma_1^2\sigma^2e^{2\gamma_1^2\mu^2}e^{\gamma_1\mu}}{(n+1)^{\gamma_1\mu}}
	\sum_{k=1}^n\frac{1}{k^{2-\gamma_1\mu}}\,.
	\end{displaymath}
	Finally, we conclude by defining $S(n)
	= (n+1)^{-\gamma_1\mu}\sum_{k=1}^n k^{\gamma_1\mu-2}$ and using
	Lemma~\ref{lemma:series} to obtain the upper bounds on $S(n)$ given in the
	theorem statement depending on the value of $\gamma_1\mu$.

\end{proof}

\begin{lemma}
	\label{lemma:series}
	For any $\alpha>0$ and $n\geq 1$:
	\begin{displaymath}
		\frac{(1+n)^{1-\alpha} - 1}{1-\alpha}
		\leq\sum_{k=1}^n\frac{1}{k^\alpha}
		\leq \frac{n^{1-\alpha} - \alpha}{1-\alpha}
		\quad\text{and}\quad
		\frac{n^{1+\alpha}}{1+\alpha}\leq\sum_{k=1}^n k^{\alpha}\leq
	\frac{(n+1)^{1+\alpha}-1}{1+\alpha}
	\,,
	\end{displaymath}
	where the first bound remains true by continuity at $\alpha=1$:
	$\log(1+n)\leq\sum_{k=1}^n\frac{1}{k}\leq 1+\log n$.
\end{lemma}

\begin{proof}
	Immediate by approximating the discrete sums from above and
	below by integrals.
\end{proof}
}

\usepackage{xcolor}
\definecolor{shadecolor}{rgb}{1,.8,.3}

\usepackage{etoolbox}

\makeatletter
\patchcmd{\@makecaption}
  {\parbox}
  {\advance\@tempdima-\fontdimen2} % decrease the width!
  {}{}
\makeatother  
\usepackage[caption = false]{subfig} 

\usepackage{array}
\usepackage{makecell}

\newcommand{\EqLabel}[1]{\quad\quad\quad\text{(#1)}}
\usepackage[colorlinks = true,
            linkcolor = blue,
            urlcolor  = blue,
            citecolor = blue,
            anchorcolor = blue]{hyperref}

\usepackage{setspace}

\newcounter{remark}[section]
\newenvironment{remark}[1][]{
\refstepcounter{remark}
\par\medskip\noindent
   \text{\textit{Remark~\thesection.\theremark.}}~#1\rmfamily}

\title{The Proximal Robbins--Monro Method}

\author{
Panos Toulis\thanks{University of Chicago, Booth School of Business; 
{\em email:}~\url{panos.toulis@chicagobooth.edu}.
}
\and 
Thibaut Horel\thanks{MIT, Laboratory for Information \& Decision Systems;
{\em email:}~\url{thibauth@mit.edu}.
}
\and Edoardo M. Airoldi\thanks{Temple University, Fox School of Business;
{\em email:}~\url{airoldi@temple.edu}.
}
}

\renewcommand{\P}{F}

\begin{document}

\maketitle

\begin{abstract}
The need for parameter estimation with massive datasets has reinvigorated interest in stochastic optimization and iterative estimation procedures. Stochastic approximations are at the forefront of this recent development as they yield procedures that are simple, general, and fast.
However, standard stochastic approximations are often numerically unstable.
% , which is only partially addressed in current practice with ad hoc techniques.
% 
Deterministic optimization, on the other hand, increasingly uses proximal updates to achieve numerical stability in a principled manner. A theoretical gap has thus emerged.
While standard stochastic approximations are subsumed by the framework of~\citet{robbins1951}, there is no such framework for stochastic approximations with proximal updates. 
In this paper, we conceptualize a proximal version of the classical Robbins--Monro procedure.
Our theoretical analysis demonstrates that the proposed procedure has important stability benefits over the classical Robbins--Monro procedure, while it retains the best known convergence rates. Exact implementations of the proximal Robbins--Monro procedure are challenging, but we show that approximate implementations lead to procedures that are easy to implement, and still dominate classical procedures by achieving numerical stability, practically without tradeoffs. 
Moreover, approximate proximal Robbins--Monro procedures can be applied even when the objective cannot be calculated analytically, and so they generalize stochastic proximal procedures currently in use. 
%
% These advantages are akin to those seen in deterministic proximal optimization, where the framework is typically impossible to instantiate exactly, but where approximate instantiations lead to new optimization procedures that dominate classical ones.
\end{abstract}


\newpage
\section{Introduction}
\label{section:introduction}
In a seminal paper,~\citet{robbins1951} considered the problem of estimating the zero $\thetastar$ of a function $h:\Reals{p}\to\Reals{}$, 
where  $h(\theta)$ is unknown but can be 
unbiasedly estimated by a function $H$ of some random variable $\xi$, such that 
$\Ex_\xi(H(\theta, \xi))=h(\theta)$, for fixed $\theta\in\Theta\subseteq\Reals{p}$.
Starting from $\theta_0$, \cite{robbins1951} 
iteratively estimated $\thetastar$ using observations $\xi_1, \xi_2, \ldots$, as follows:
\begin{align}
\label{eq:rm}
\thetarm{n} = \thetarm{n-1} - \gamma_{n} H(\theta_{n-1}, \xi_n),
\end{align}
where 
% $\gamma_n$ is known as the learning rate sequence, typically defined as 
$\gamma_n\propto1/n$, for $n=1, 2, \ldots$, so that $\sum \gamma_i^2 < \infty$ 
and $\sum\gamma_i = \infty$. 
% The former condition guarantees convergence, and the latter guarantees that convergence can be towards any point in $\Reals{p}$.
%
\citet{robbins1951} proved convergence in quadratic mean for the procedure in~Equation~\eqref{eq:rm}, under a monotonicity assumption for $h$
and bounded second moments for the noise,~$H(\theta, \xi)- h(\theta)$.
\citet{blum1954,ljung1992stochastic, kushner2003stochastic, borkar2008stochastic} later strengthened this convergence result.
Due to its remarkable simplicity and empirical performance, the Robbins--Monro procedure has found widespread applications across scientific fields, 
including statistics~\citep{nevel1973stochastic, ruppert1988efficient}, engineering~\citep{benveniste1990adaptive}, and optimization~\citep{nesterov2004introductory}.

Recently, the Robbins--Monro procedure has attracted considerable interest in
machine learning with large data sets~\citep{zhang2004solving, bottou2010,
	moulines2011non, bottou2016optimization}, and in scalable statistical inference~\citep{toulis2015scalable, chen2016statistical, su2018statistical,
	li2017statistical, toulis2017aos}. 
In this context, given a dataset $D$, the Robbins--Monro
	procedure in Equation~\eqref{eq:rm} can be applied with $h(\theta)$ being
	the gradient of the negative log-likelihood of $\theta$ given $D$ and $H(\theta, \xi)$
	being the gradient of the negative log-likelihood of $\theta$ calculated at
	a single data point sampled with replacement from $D$. Standard theory then
	implies that $\thetarm{n}$ converges to a point $\theta_\infty$ for which
	$h(\theta_\infty)=0$. In other words, $\thetarm{n}$ converges to the
	maximum-likelihood estimator (or maximum a posteriori if regularization is
	used) given dataset $D$.  In this context, $h$ is the gradient of a convex
	scalar potential, and the Robbins--Monro procedure is commonly referred to
	as stochastic gradient descent~(SGD). 
% While this context is widely applicable, and also helps with concreteness and interpretation, in this paper we consider the more general case where $h$ is not necessarily the gradient of a convex potential~(with the exception of heorem~\ref{theorem:deviance}).
% Thus, our present work is placed in the broader context of stochastic approximation.

% In this work, we will study the theoretical properties of a modified
% stochastic approximation method in the more general setting with a stream of
% data points, but will also discuss applications to iterative estimation with
% a finite data set.

%% Issue 
A well-known issue with the Robbins--Monro procedure, however,
is numerical stability and sensitivity to specification of hyperparameters, especially
the learning rate $\gamma_1$.
% that the learning rate sequence crucially affects  both its numerical stability and convergence.
% non-asymptotic convergence
% Regarding stability, consider a simple example where $\thetarm{n}$  in Eq.\eqref{eq:rm} is approximating a scale parameter,  and thus needs to be positive. Clearly, even if the previous iterate $\thetarm{n-1}$ is positive,  there is no guarantee that the update in Eq.\eqref{eq:rm} will provide a positive next iterate $\thetarm{n}$.
%
For instance, the procedure can 
be arbitrarily slow  if $\gamma_1$ is even slightly misspecified.
To illustrate, suppose that $\gamma_n = \gamma_1/n$, and 
there exists a scalar potential $\P$, such that $\nabla \P(\theta) = h(\theta)$, for all $\theta\in\Theta$. 
If $\P$ is strongly convex with parameter $\mu$,
then  $\Ex{\|\thetarm{n}-\thetastar\|^2}= \bigO{n^{-\epsilon}}$ if $\epsilon=2\mu\gamma_1< 1$ \cite[Section 1]{nemirovski2009robust}; \cite[Section 3.1]{moulines2011non}.
On the other hand, the procedure can diverge even in the first few iterations if 
the learning rate is too large, 
especially with non-Lipschitz likelihoods as in Poisson regression~\citep{toulis2014statistical}.
In summary, small learning rates can make the Robbins--Monro iterates converge very slowly, whereas large learning rates can make the iterates diverge numerically. 
Importantly, the requirements for numerical stability and fast convergence are very hard to reconcile in practice, especially in large-scale problems, which renders 
the Robbins--Monro method, and all its derived procedures, 
inapplicable without extensive heuristic modifications~\citep{bottou2012stochastic}.

%%%%
%%%%
%%%    The new SA algorithm
\section{The proximal Robbins--Monro procedure: An overview}
\label{sec:implicit}
%

% prox updates
In this paper, our idea to improve the stability of the Robbins--Monro procedure is 
to leverage the proximal point algorithm of~\citet{rockafellar1976monotone}. Assuming a potential function $\P$ and a current iterate $\theta_{n-1}$, the proximal point update is defined as follows:
\begin{align}
\label{eq:proximal}
\thetaMed{n} = \mathrm{prox}_{\gamma_n F}(\theta_{n-1}) :=
\arg\min_{\theta\in\Theta}\left\{\frac{1}{2\gamma_n}\|\theta-\thetaim{n-1}\|^2 + \P(\theta)\right\} 
= \theta_{n-1} - \gamma_n h(\theta_n^+).
\end{align}
In recent years, interest in optimization through
proximal operators~(i.e., function $\mathrm{prox}_{\gamma_n F}$ above) has exploded because the resulting proximal procedures are stable and 
converge with minimal assumptions~\citep{bauschke2011convex,
parikh2013proximal}.  In addition, they can be applied to settings where the
objective function is the sum of a smooth and a non-smooth function
(as is common when using regularization), and often lead to
efficient, parallelizable algorithms.

% stability
To illustrate the stability of proximal updates let us take norms in Equation~\eqref{eq:proximal}:
$$
\|\theta_n^+-\thetastar\|^2 + 2\gamma_n (\theta_n^+-\theta_\star)^\top h(\theta_n^+) + 
\gamma_n^2 \|h(\theta_n^+)\|^2 = \|\theta_{n-1}-\thetastar\|^2.
$$
By convexity of $\P$, we have $h(\theta)^\top(\theta-\thetastar)\ge0$ for
any $\theta$, and so unless $h(\thetaMed{n})=0$ we obtain $\|\thetaMed{n}-\thetastar\|^2
< \|\thetaim{n-1}-\thetastar\|^2$, a contraction.
More generally, the proximal update can be
shown to be a firmly non-expansive operator for any choice of $\gamma_n >0$, and
so the procedure in Equation~\eqref{eq:proximal} is indeed numerically stable.
% in practice..
In practice, however, the proximal point algorithm is infeasible as solving the minimization problem \eqref{eq:proximal} is, in general, as hard as minimizing $F$ directly. Regardless, classical results show that approximate solutions to Equation~\eqref{eq:proximal} are still stable 
if the approximation errors are small enough~\citep{rockafellar1976monotone}.

In this paper, we first introduce stochastic errors in the proximal point updates, and study the properties of the resulting procedure from a probabilistic viewpoint.
This procedure will follow a stylized model at first, but will later serve as a template for concrete, approximate implementations.
To that end, we begin with the following stylized model: an agent has an initial estimate $\theta_0$, 
then an oracle calculates the proximal update $\theta_1^+$ according to Equation~\eqref{eq:proximal}, 
which the agent then observes with error $\varepsilon_1$, as $\theta_1 = \theta_1^+ - \gamma_1 \varepsilon_1$; 
then, the oracle computes the proximal update $\theta_2^+$ given $\theta_1$, and so on.
This procedure is depicted in Table~\ref{tab:spp}, and is also summarized below:
\begin{table}
\caption{Stylized model of the Stochastic Proximal Point Algorithm.
The update from $\theta_{n}$ to $\theta_n^+$ is deterministic, and  
from $\theta_n^+$ to $\theta_{n+1}$ it is stochastic.\label{tab:spp}
}
\begin{tabular}{llllllllllll}
 &  &  &   &  &  &  &  &  &   &  &  \\
agent  & $\theta_0$ &  &   &  & $\theta_1$ &  & $\theta_{n-1}$ &  &   &  & $\theta_n$ \\
       &   & $\searrow$ &   & $\nearrow$~(error) &   & $\ldots$ &   & $\searrow$ &   & $\nearrow$~(error) &   \\
oracle &   &  & $\theta_1^+$ &  &   &  &   &  & $\theta_n^+$ &  &  
\end{tabular}
\end{table}
\begin{align}
\theta_n^+ & = \theta_{n-1} - \gamma_n h(\theta_n^+), \label{eq:spp0}\\
\theta_n & = \theta_n^+ - \gamma_n \varepsilon_n.\EqLabel{Stochastic Proximal Point Algorithm}
\label{eq:spp}
\end{align}
An immediate concern with the {\em stochastic proximal point algorithm} of Equation~\eqref{eq:spp} is whether it inherits
the stability properties of the classical proximal point algorithm.
Indeed, in Section~\ref{section:theory} we show that when $\gamma_n$ decreases at a proper rate, and $\varepsilon_n$ is random error with uniformly bounded variance, then the stochastic proximal point algorithm converges, and is numerically stable. 
In particular, we prove results on almost sure convergence~(Theorem~\ref{theorem:convergence}), and derive error bounds for convex~(Theorem~\ref{theorem:deviance})
and strongly convex objectives~(Theorem~\ref{theorem:nonasymptotic}). We then discuss the stability of the algorithm by analyzing the dependence of 
expected errors, $\Ex(\|\theta_n-\theta_\star\|^2)$, on the initial error, $\Ex(\|\theta_0-\theta_\star\|^2)$, and the learning rate, $\gamma_1$; see Section~\ref{sec:non_asymptotics} for details.

Following our theoretical analysis of the stylized model of stochastic proximal points, we then focus on concrete instantiations.
Our key assumption is that $h(\theta)$ can be estimated by some random variable $H(\theta, \xi)$, such that $\Ex_\xi(H(\theta, \xi)) = h(\theta)$ and $\varepsilon=
 H(\theta, \xi) - h(\theta)$.
These definitions together with~\eqref{eq:spp0} and~\eqref{eq:spp} imply the following procedure:
\begin{align}
\label{eq:proxRM}
\theta_n = \theta_{n-1} - \gamma_n H(\theta_n^+, \xi_n).\EqLabel{Proximal Robbins--Monro Procedure}
\end{align}
This is a form of a {\em proximal Robbins--Monro procedure}, since its update in Equation~\eqref{eq:proxRM} only differs from the classical Robbins--Monro update of Equation~\eqref{eq:rm} in using the proximal update $\theta_n^+$ instead of $\theta_{n-1}$ in $H(\theta, \xi)$. As a 
special case of the stochastic proximal point algorithm defined above, the proximal Robbins--Monro is also numerically stable. However, its exact  implementation  remains problematic due to the presence of the proximal term $\theta_n^+$ in Equation~\eqref{eq:proxRM}.%, which is hard to compute.

We thus propose and study two different ways to implement the proximal Robbins--Monro procedure, 
depending on whether we can observe $\xi$ directly, or not.
First, we consider settings where $\xi$ can be observed directly.
For example, in the context of stochastic gradient descent~(see Section~\ref{section:plugin}), $\xi$ is a random datapoint from the dataset, and $H(\theta, \xi)$ is the corresponding stochastic gradient calculated at parameter value $\theta$. 
In such settings, we can perform an interesting --- and perhaps counterintuitive --- application of the ``plug-in principle''. 

In particular, taking expectations in Equation~\eqref{eq:proxRM} yields
$
\Ex(\theta_n | \Fn{n-1}) = \theta_n^+,
$
where $\Fn{n-1}$ is the natural filtration, $\sigma(\xi_1, \ldots, \xi_{n-1})$. Since $\theta_n$ is an unbiased 
estimator of $\theta_n^+$ we can simply plug in $\theta_n$ on the right-hand side of~\eqref{eq:proxRM} to obtain:
\begin{align}\label{eq:intro_isgd}
\theta_n = \theta_{n-1} - \gamma_n H(\theta_n, \xi_n).\EqLabel{Incremental Proximal/Implicit Methods}
\end{align}
Equation~\eqref{eq:intro_isgd} describes a wide family of stochastic optimization methods known as incremental proximal methods, 
or implicit stochastic gradient descent, depending on the field of application. We discuss related work in Section~\ref{section:related}, 
and provide details and practical examples in Section~\ref{section:impl_plugin}.

Second, we consider the more novel and challenging setting, 
where we cannot sample or observe $\xi$ directly.
This is the case, for example, when we want to be agnostic about the 
analytical form of $h(\theta)$, or when $H(\theta, \xi)$ can only be sampled successively in the context of a sequential experiment~(see Section~\ref{section:example}).
In such cases, stochastic gradients or implicit updates as in Equation~\eqref{eq:intro_isgd} are not applicable.
To address this challenge we first take a full expectation in Equation~\eqref{eq:proxRM} 
to obtain $\Ex(\theta_n^+ - \theta_{n-1} + \gamma_n H(\theta_n^+, \xi_n) \big | \Fn{n-1}) = 0$.
This implies that conditional on $\Fn{n-1}$ we can view the proximal iterate $\theta_n^+$ as a solution to the following characteristic equation:
$$
\Ex\big(\theta - \theta_{n-1} + \gamma_n H(\theta, \xi_n)\mid \Fn{n-1}\big) = 0.
$$
The key idea is then to apply the classical Robbins--Monro procedure directly on this characteristic equation, 
leading to the following stochastic fixed point procedure:
\begin{align}
w_1 & = \theta_{n-1},\nonumber\\
w_k & = w_{k-1} - a_k \big (\gamma_n H(w_{k-1}, \xi_k) + w_{k-1} - w_1\big),\quad k=1, \ldots, K.\nonumber\\
\theta_n  & = w_K.\EqLabel{Proximal Stochastic Fixed-Point}\label{eq:intro_sfp}
\end{align}

At first, it may seem that the {\em proximal stochastic fixed-point procedure} described  in Equation~\eqref{eq:intro_sfp} may have the same stability issues as classical stochastic approximation, since it is using classical updates in the inner loop. To investigate this, in Section~\ref{section:impl_nested} we analyze the convergence properties of this procedure, which is particularly challenging due to its nested structure. 
Our analysis reveals conditions under which the procedure can be more stable 
than classical Robbins--Monro.  In Section~\ref{section:example}, we also show significant benefits in numerical stability using the classical quantile regression example of~\citet{robbins1951}.
The proximal stochastic fixed-point procedure of Equation~\eqref{eq:intro_sfp} and its theoretical analysis, therefore constitute a key contribution of this paper.  We are unaware of other proximal procedures, exact or approximate, that address settings where the random components, $\xi$, cannot be observed directly, or where the underlying procedure is comprised of nested stochastic fixed points.  

The rest of this paper is structured as follows. In Section~\ref{section:related} we discuss related work. In Section~\ref{section:theory}, we study the stochastic proximal point algorithm, and show theoretically its appealing statistical and stability properties. 
We then focus on the proximal Robbins--Monro procedure as a natural instance of the stochastic proximal point algorithm.
As mentioned earlier, proximal Robbins--Monro is also an ``idealized procedure", i.e., 
it is well-defined mathematically but, in general, it cannot be directly computed.
% Such idealized procedures are ubiquitous in statistics and optimization.  For example, the proximal point algorithm described above is approximated in practice even though it has a concrete mathematical definition. 
%
We thus explore the extent to which its theoretical properties carry through to the 
two approximate implementations outlined above.
Specifically, in Section~\ref{section:impl_plugin}, we 
discuss the implicit procedures described in Equation~\eqref{eq:intro_isgd}, with concrete examples, 
and in Section~\ref{section:impl_nested} we analyze the stochastic fixed-point procedure of Equation~\eqref{eq:intro_sfp} with a full convergence analysis. 
To illustrate our theory, we present empirical results on 
both approximate implementations in Section~\ref{sec:simulations} 
and Section~\ref{section:example}, all showing the stability benefits of our approach.

%
%\remark{
%As mentioned above, the Proximal Robbins--Monro we study in this paper is an ``idealized procedure", i.e., 
%it is well-defined mathematically but, in general, it cannot be directly computed.
%Such idealized procedures are ubiquitous in statistics and optimization. 
%For example, the proximal
%point algorithm described above is approximated in practice even though it has a concrete mathematical definition. 
%% In this paper, the study of our suggested procedure follows an outline similar to how other idealized procedures are studied: first, we explore the properties of the idealized procedure as if it were computed exactly, and then explore the extent to which these properties carry through to approximate implementations.
%}

%This resolves the conflicting
%requirements for stability and convergence in classic stochastic approximation, and provides valuable flexibility in choosing the learning rate $\gamma_n$. The theoretical results of the following section 
%do confirm that implicit stochastic approximation 
%converges under minimal conditions, 
%maintains the asymptotic properties 
%of classic stochastic approximation, but also 
%has remarkable stability in short-term iterations.
%

\subsection{Related work and contributions}
\label{section:related}

There is voluminous literature on classical stochastic approximation. The early mathematical work 
by~\citet{robbins1951, sacks1958asymptotic, fabian1968asymptotic, nevel1973stochastic, robbins1985convergence, wei1987multivariate} established the fundamental properties, including convergence and asymptotic laws.
This work was subsequently pivotal in engineering, and particularly in systems identification and tracking~\citep{ljung1992stochastic,  benveniste1990adaptive};~see also the excellent review by~\citet{lai2003stochastic}.
More recently, there have been important developments in studying 
stochastic approximations through the lens of dynamical systems theory, spearheaded by~\citet{kushner2003stochastic} and~\citet{borkar2008stochastic}.
Roughly at the same time, stochastic approximations appeared in machine learning, usually in the form of stochastic gradient descent~(SGD) methods, and especially
in applications with large data sets and complex models~\citep{zhang2004solving, bottou2010}.

%
%However, classical stochastic approximations are  numerically unstable, and
%often impossible to apply  without extensive heuristics. In this
%paper, we introduce an implicit stochastic approximation procedure, defined in
%Equations~\eqref{eq:proxRM} and~\eqref{eq:proxRM_fp}, which aims at
%mitigating stability problems of classical approximation through proximal
%updates.  In the same way that the classical method of~\citet{robbins1951} is the stochastic analog of gradient descent
%in deterministic optimization, the method of implicit stochastic approximation, proposed in this 
%paper, is the
%stochastic analog of the proximal point
%algorithm~\citep{rockafellar1976monotone}.
%%  which is the quintessential methodin proximal optimization.  
%This fills a crucial gap in the literature of
%stochastic approximations.
%As visualized in Table~\ref{table:related}, the
%method we introduce is general enough to cover both cases where an analytic
%form of the objective function is known, and cases where no such form is known for the objective or its gradients.~\rev{For instance, in statistical estimation the likelihood is usually known, and 
%stochastic approximation can be used for scalable estimation. 
%However, in a sequential experiment we usually aim to optimize an objective function that is unknown~(e.g., optimize drug dosage), and the objective function values can only be queried sequentially.
%}

There are mainly two lines of literature that are directly related to our work, as depicted in Table~\ref{table:related}. 
In one line of work, the proximal update is deterministic and is performed after a classical stochastic update. 
For example, the forward-backward procedure of~\citet{singer2009efficient} and the proximal stochastic gradient procedure studied by~\citet{rosasco2014convergence, rosasco2016stochastic, bianchi2016dynamical} fall into this category---see Section~\ref{section:theory} for a related discussion.
Such procedures
first make the update $\tilde\theta_n = \theta_{n-1} - \gamma_n H(\theta_{n-1}, \xi_n)$, and then define $\theta_{n} = \mathrm{prox}_{\gamma_n f}(\tilde\theta_n)$, where $f$ is some convex regularization function.  
In our work, we wish to avoid making an explicit update, if possible, to ensure stability. 
A notable exception is presented in Section~\ref{section:impl_nested}, where we 
discuss the stochastic fixed-point procedure in Equation~\eqref{eq:intro_sfp}. 
This procedure involves multiple explicit updates within a nested procedure, which, however, do not introduce instability thanks to the problem structure.

Another line of work involves procedures as in Equation~\eqref{eq:intro_isgd}, where implicit updates are directly used in the update equation. Incremental proximal procedures~\citep{bertsekas2011incremental}, and implicit SGD~\citep{toulis2014statistical, toulis2017aos} fall into this category. 
The implicit update in~\eqref{eq:intro_isgd} can be solved efficiently 
in many statistical models~\citep[Algorithm 1]{toulis2014statistical}, including generalized linear models.
In numerical optimization and engineering, 
the stochastic proximal point algorithms studied by~\citet{bianchi2016ergodic, patrascu2017nonasymptotic, patrascu2019stochastic} are closely related, and thus 
different from our definition in~\eqref{eq:spp}.
% despite the naming resemblance, these should not be confused with the stochastic proximal algorithm of~\citet{rosasco2014convergence}.
%
Interestingly, all such procedures can be viewed as the plug-in versions of the proposed proximal Robbins--Monro in Equation~\eqref{eq:proxRM}.
%  We discuss this idea further in Section~\ref{section:impl_plugin}.

\renewcommand{\arraystretch}{1}
\begin{table}
\footnotesize
\caption{
\label{table:related}
\footnotesize {\em Depiction of related work. Modern procedures, such as SGD, are instantiations of the classical Robbins--Monro procedure~\citep{robbins1951}. 
% This method was initially conceptualized as a stochastic analog of root finding methods, such as Newton-Raphson or gradient descent. 
The proximal Robbins--Monro procedure we study in this paper leads to well-known implicit procedures, 
and also to novel procedures which can work 
even when the random component $\xi$ of the stochastic approximation cannot be observed directly.}
\vspace{5px}}
\begin{tabular}{ l | c c }
& \multicolumn{2}{c}{{\bf Solve:}~$\Ex_\xi(H(\theta, \xi)) = 0$.} \\
\makecell{\bf samples\\\bf $H(\theta, \xi)$}  & 
	\makecell{{\bf Classical Robbins--Monro}  \\
	$ \theta_n = \theta_{n-1} - \gamma_n H(\theta_{n-1}, \xi_n)$} & 
	\makecell{{\bf Proximal Robbins--Monro} \\ $ \theta_n = \theta_{n-1} - \gamma_n H(\theta_n^+, \xi_n)$} \\ 
\hline
 \makecell{can observe \\$\xi$ directly} & 
 		\makecell{stochastic gradient descent\\
  		\citep{coraluppi1969stochastic};
		\citep{zhang2004solving};\\
		\citep{bottou2010};
		natural gradients~\small\citep{amari1998natural};\\
		adaptive gradients~\citep{duchi2011adaptive}} & 
		\makecell{implicit stochastic gradients\small\\
  ~\citep{bertsekas2011incremental};\citep{bianchi2016ergodic}\\
  \citep{toulis2017aos};\\
  stochastic proximal gradients\\
  ~\citep{singer2009efficient};\\
  \citep{rosasco2014convergence}
  }
  \\  
 \makecell{cannot observe\\ $\xi$ directly} & \makecell{structural breaks/tracking, quantile estimation
   \\~\citep{benveniste1990adaptive}; \citep{robbins1951}}
 & \makecell{Prox-Stochastic Fixed Point\\ (Equation~\eqref{eq:spp},~Section~\ref{section:impl_nested})} 
\end{tabular}
\end{table}

%
%\begin{figure}[t!]
%\begin{center}
%\resizebox{14.5cm}{!}{\input{related_work2.tikz}}
%\end{center}
%\label{fig:related}
%	\caption{Graphical depiction
%	of related work. Modern popular procedures, such as stochastic gradient
%descent, are instantiations of the classical stochastic approximation method of Robbins-Monro~\citep{robbins1951}. 
%	The Robbins-Monro method was initially conceptualized as a stochastic analog of root finding methods, such as Newton-Raphson or gradient descent. Our work provides a stochastic approximation method  with proximal updates. Instantiations of our method include well-known existing procedures that employ proximal updates, such as implicit stochastic gradient descent. Additionally, it leads to novel procedures 
%with nested stochastic approximations, which can be applied even in cases where neither the objective function nor its gradient are known analytically.}
%\end{figure}

From a theoretical perspective, the central contribution of this paper is first the introduction of 
the proximal Robbins--Monro procedure as the stochastic analog of the classical proximal point algorithm.
This procedure differs from classical 
Robbins--Monro procedures by using the proximal point, $\theta_n^+$, in its iterations.
We provide a full analysis of the convergence properties of the new procedures in Section~\ref{section:theory}. 
This fills a gap in the literature that has remained open since classical stochastic approximation was introduced by~\citet{robbins1951} as the stochastic analog of gradient descent.
Our analysis shows that the proximal Robbins--Monro procedure is more stable numerically than classical Robbins--Monro, and is also less sensitive to 
hyperparameter tuning in achieving the best known convergence rates.
% for non-averaged stochastic approximation.%

From a practical perspective, the proximal Robbins--Monro procedure is generally infeasible.
We thus develop two approximate implementations.
First, in Section~\ref{section:impl_plugin} we discuss an implementation based on the plug-in principle, which leads to Equation~\eqref{eq:intro_isgd},
 and a large family of implicit SGD procedures~\citep{bertsekas2011incremental, toulis2014statistical}. These procedures are becoming increasingly popular thanks to their numerical
			stability compared to classical SGD.
			They are also easy to implement in a broad family of
			models, and their theoretical properties are now well understood~\citep{kulis2010implicit, bertsekas2011incremental,
			toulis2017aos, ryu2014stochastic, toulis2016towards, patrascu2017nonasymptotic, patrascu2019stochastic, asi2019importance}.

Second, in Section~\ref{section:impl_nested} we discuss 
an implementation of the proximal Robbins--Monro procedure based on the fixed point procedure of Equation~\eqref{eq:intro_sfp}. 
Importantly, this procedure can operate even when we cannot observe $\xi$ directly, which includes settings where $h(\theta)$ is not known analytically. We present a full convergence analysis of the procedure, which is particularly challenging due to its nested structure. 
In Section~\ref{section:example}, we also illustrate significant benefits in numerical stability through the classical quantile regression example of~\citet{robbins1951}.
As mentioned earlier, the fixed-point procedure and its theoretical analysis constitute a key contribution of this paper. 
We are unaware of other proximal methods, exact or approximate, that 
can be applied to pure stochastic approximation settings.
% , i.e.,  settings where $h(\theta)$ is analytically unknown, or can  only be sampled sequentially through experiments.
%where the underlying procedure is comprised of nested stochastic fixed points.
Stochastic fixed point procedures exist in the classical literature~\citep[e.g., Section 10.2]{borkar2008stochastic} but they usually lack a standard nonasymptotic error analysis~(similar to Theorem~\ref{thm:nested} in this paper), and so their stability properties are unknown.
%

%%%%  THEORY    %%%%%
\section{Theory of the stochastic proximal point algorithm}
\label{section:theory}

In this section, we analyze theoretically the stochastic proximal point algorithm in Equation~\eqref{eq:spp}.
Specifically, we study convergence~(Section~\ref{sec:asymptotic}),
asymptotic normality (Section~\ref{sec:normality}), and non-asymptotic
convergence rates (Section~\ref{sec:non_asymptotics}). 
We emphasize that the analysis here applies directly to the proximal Robbins--Monro procedure of Equation~\eqref{eq:proxRM}.
Later, we show that the theoretical properties studied here, and especially those that relate to numerical stability, pass onto the
approximate implementations of proximal Robbins--Monro.
All proofs can be found in Appendix~\ref{appendix:main}.
We need some notation first.~\StateAssumptions\

\assumeConvex\ is implied by convexity of $F$ but is weaker since monotonicity of its gradient $h$ is only required to hold at $\thetastar$. \assumeConvexStrict\ states that $F$ is strictly convex at $\thetastar$ (in particular, it implies that $\thetastar$ is unique).
\assumeStrongConvex\ is implied by strong convexity of $F$ but is weaker since, similarly to \assumeConvex, the quadratic lower bound is only required to hold with respect to $\thetastar$.
%
%Assumption~\aPotential\ will be used later, in Sections~\ref{section:impl_plugin} and~\ref{section:impl_nested}, where we consider concrete implementations of our procedure.
%
\assumeErrors\ was introduced by \citet{robbins1951},
and has since been standard in stochastic approximation analysis.
It simply states that the stochastic errors in the observations of $h$ have zero mean and uniformly bounded variances.
It could be weakened to include slowly growing errors, $\sigma_n^2$, provided that 
$\sum_{i=1}^\infty \sigma_i^2 \gamma_i^2 < \infty$.
%since bounded noise is a crucial condition for convergence. (Tbo: commented out since it is not true)
%
\assumeMain\assumeLind\ is the Lindeberg condition that is used to prove
asymptotic normality of $\thetaim{n}$, later in this section.

Overall, our assumptions are weaker than the assumptions in classical stochastic approximation because they refer to the idealized procedures of Equation~\eqref{eq:spp} and Equation~\eqref{eq:proxRM}; compare, for example, 
Assumptions \aGn--\aLind\ with assumptions (A1)--(A4) of \citet[Section 2.1]{borkar2008stochastic}, or the assumptions by \citet[Theorem 15]{benveniste1990adaptive}.
In comparison to forward-backward procedures~\citep[e.g.]{rosasco2016stochastic, bianchi2016dynamical}, we share common assumptions on Lipschitzness of the regression function $h$~(\assumeLip) and bounded second moments for the noise term~(\assumeErrors).
The main difference is that forward-backward procedures require certain ``fine-tuning" conditions for the learning rate. For example, assumption (A2) of ~\citet{rosasco2016stochastic} requires that 
$\gamma_n$ decays sufficiently fast with respect to the noise level in $\varepsilon$. Our procedure does not require such assumptions because the forward-backward steps are transposed~(the implicit step happens first), which adds numerical stability, as shown in Theorem~\ref{theorem:nonasymptotic}.
In some sense, our procedure is a form of ``backward-forward splitting".

\subsection{Convergence of stochastic proximal points}
\label{sec:asymptotic}
In Theorem~\ref{theorem:convergence}, we derive a proof of almost sure convergence of the stochastic proximal point algorithm, which mainly relies on the supermartingale lemma of \cite{robbins1985convergence}. 

\TheoremConvergence

The conditions for almost sure convergence of the stochastic proximal point are weaker than classical stochastic approximation. 
For example, in classical stochastic approximations
it is typically assumed that the iterates $\thetarm{n}$ are almost surely bounded.
See, for example, Assumption (A4) of \citet{borkar2008stochastic}.

%
% Non-asymptotic analysis
% 
\subsection{Non-asymptotic analysis}
\label{sec:non_asymptotics}

In this section, we derive upper bounds for deviance of the potential function, 
$\Ex(\P(\thetaim{n}) - \P(\thetastar))$, 
and the mean squared errors, $\Ex{\|\thetaim{n}-\thetastar\|^2}$.
This provides information on the rate of convergence, as well as 
the stability of the stochastic proximal point algorithm.
Theorem \ref{theorem:deviance} on deviance assumes non-strong convexity of $\P$, whereas Theorem \ref{theorem:nonasymptotic} on squared error assumes strong convexity.
\TheoremDeviance

There are two main results in Theorem \ref{theorem:deviance}.
First, the rates of convergence for the deviance of the stochastic proximal point algorithm are either
$\bigO{n^{-1+\gamma}}$ or $\bigO{n^{-\gamma/2}}$, 
depending on the learning rate, $\gamma$.
Second, there is a uniform decay of expected deviance 
towards zero,
% , since the constants $n_{0,1}, n_{0,2}, n_{0,3}$ can be made small, depending on what accuracy is desired.
%  desired accuracy  in the constants of the upper bounds in Theorem \ref{theorem:deviance}.
%
whereas in standard stochastic approximation under 
non-strong convexity, there is a term of the form $\exp(4L^2\gamma_1^2n^{1-2\gamma})$ \citep[Theorem 4]{moulines2011non}, which can 
amplify the initial conditions arbitrarily.
Thus, the stochastic proximal point algorithm, and consequently the proximal Robbins--Monro procedure, have 
similar asymptotic properties to classical stochastic 
approximation, but they are more stable numerically, and less sensitive to initial conditions or 
hyperparameter tuning.

\remark{The best rate of convergence for the proximal Robbins--Monro as shown in Theorem~\ref{theorem:deviance} is $O(n^{-1/3})$,
which  matches the best known rate for classical stochastic approximations with non-strongly convex objective~\citep[Theorem 4]{moulines2011non}. 
This rate is suboptimal since it is worse than the minimax rate of $O(n^{-1/2})$ that is achieved through Polyak-Ruppert averaging~\citep{ruppert1988efficient}.
We conjecture that our proposed procedure can also achieve the minimax rate through averaging,
but we leave this for future work.
}

\remark{
The proof of Theorem~\ref{theorem:deviance} presents some unique technical challenges, 
including an implicit inequality of the form $b_n + g(b_n) \le b_{n-1}$, with $g$ being a non-explicit, non-decreasing function.
Our strategy is to solve the reverse recursive inequality, $\tilde b_n(\beta) + g(\tilde b_{n}(\beta)) \ge \tilde b_{n-1}(\beta)$, 
in some parametric family, such as $\tilde b_n(\beta) = O(n^{-\beta})$, which is more tractable.
Then, it is easy to show that $\tilde b_n(\beta)$ is an upper bound for $b_n$, for any $\beta$.
Thus, a natural upper bound for $b_n$ is given by $b_n \le \arg\min_\beta \tilde b_n(\beta)$.
This solution strategy is reminiscent of the majorization-minorization idea~\citep{lange2010numerical}, 
and may be more broadly useful.
}

%

%% Main non-asymptotic theorem
\TheoremMSE

There are two main results presented in Theorem \ref{theorem:nonasymptotic}.
First, the rate of convergence for the expected errors,
$\Ex(\|\thetaim{n}-\thetastar\|^2)$, is $\bigO{n^{-\gamma}}$, 
which matches the rate of convergence for classical stochastic approximation under 
strong convexity \citep[Theorem 22]{benveniste1990adaptive}.
The best possible rate here is $O(1/n)$, which is also the minimax rate 
with strongly convex objectives.
Second, there is an exponential discounting of 
initial conditions, $\zeta_0$, regardless of the specification of the 
learning rate parameter $\gamma_1$ and the Lipschitz parameter $L$.
Another way to express this, is to consider the function $\omega_n = \log(\zeta_n/\zeta_0)$ under a noise-free setting~($\sigma^2=0$).
By studying this function with respect to $\gamma_1$ (and other problem parameters, such as convexity) we can study stability.
In particular, Theorem~\ref{theorem:nonasymptotic} shows that $\omega_n=-\log (1+2\gamma_1\mu) n^{1-\gamma}$.
In contrast, in classical stochastic approximation,
$\omega_n= L^2\gamma_1^2n^{1-2\gamma} - O(n^{1-\gamma})$, which can make 
the approximation diverge numerically 
if $\gamma_1$ is even slightly misspecified with respect to $L$
\citep[Theorem 1]{moulines2011non}.
Thus, as in the non-strongly convex case of
Theorem \ref{theorem:deviance}, the stochastic proximal point algorithm
has similar asymptotic rates to classical stochastic approximation, but is also more stable.
%
% We also note that the error bounds in Theorem~\ref{theorem:nonasymptotic}
% can be used to derive deviance bounds, in addition to Theorem~\ref{theorem:deviance}.

\remark{
When $\gamma=1$, misspecification of the learning rate parameter 
can indeed lead to arbitrary slowdown to a rate $O(\max\{n^{-1}, n^{-\log \kappa}\})$.
This is also true for classical stochastic approximation~\citep[Theorem 1]{moulines2011non},
and is generally a feature of stochastic first-order methods.
The key difference between the two procedures, as described above, is numerical stability.}
%%% Asymptotic normality

\subsection{Asymptotic normality}
\label{sec:normality}

Asymptotic distributions are well studied in classical stochastic approximation. Starting from~\citet{fabian1968asymptotic} there has been extensive work in identifying asymptotic distribution laws in
stochastic approximation.
In this section, we
leverage this theory to show when iterates from stochastic proximal point procedures 
can also be asymptotically normal.
The following theorem establishes this result 
using Theorem 1 of~\citet{fabian1968asymptotic};
see also \cite[Chapter II.8]{ljung1992stochastic}.
\TheoremNormality

Theorem~\ref{theorem:normality} shows that the asymptotic distribution of $\thetaim{n}$ 
is identical to the asymptotics of the classical Robbins--Monro procedure~\citep[for example]{fabian1968asymptotic}.
Intuitively, in the limit as $n$ grows, we have that $\thetaMed{n} \approx \thetaim{n-1} + \bigO{\gamma_n}$ with high probability, 
and thus the stochastic proximal point behaves like the classical approximation procedure.
%

%%% IMPLEMENTATIONS
%%
%%
%%
\section{The proximal Robbins--Monro procedure}
\label{section:impl_plugin}

In the following sections, we focus on the special case of the stochastic proximal point algorithm, where an 
unbiased estimate $H(\theta, \xi)$ of $h(\theta)$ is available, such that 
$\Ex_\xi(H(\theta, \xi)) = h(\theta)$. This leads to the proximal Robbins--Monro 
procedure introduced in Equation~\eqref{eq:proxRM}, which we repeat here:
\begin{equation}\label{eq:proxRM2}
\theta_n = \theta_{n-1} - \gamma_n H(\theta_n^+, \xi_n).
\end{equation}

This procedure is still infeasible due to the proximal term, $\theta_n^+$, 
and so we consider approximate implementations.
Specifically, we consider two different implementations depending on whether we 
have direct access to samples of $\xi$ or not.
The former leads to well-known stochastic procedures, and so our discussion will be relatively short.
Later, in Section~\ref{section:impl_nested}, we focus on the more challenging setting where 
we cannot directly sample $\xi$, and analyze the resulting procedures in more detail.
 % In this section, we consider an implementation following from the plug-in principle, which leads to a very practical estimation procedure,  known as implicit stochastic gradient descent. Westress that such an implementation is possible even though the regressionfunction $h(\theta)$ is not known or cannot be computed.
  
  \subsection{Approximate implementation with the plug-in principle}\label{section:plugin}
As mentioned earlier, when we can observe $\xi$ directly, we can apply the plug-in principle to
 implement the proximal Robbins--Monro update in~\eqref{eq:proxRM2}.
 Specifically, by definition of the proximal update in Equation~\eqref{eq:spp} and Assumption~\ref{assumption:errors} we have:
 $
 \Ex(\theta_n | \Fn{n-1}) = \theta_{n-1} - \gamma_n h(\theta_n^+) = \theta_n^+.
$
Plugging-in $\theta_n$ for $\theta_n^+$ in Equation~\eqref{eq:proxRM2} yields:
\begin{align}
\label{eq:implementation_b}
\thetaim{n} = \thetaim{n-1} - \gamma_n H(\thetaim{n}, \xi_n).
\end{align}

The iterate $\theta_n$ appears on both  sides of Equation~\eqref{eq:implementation_b}, and the resulting 
implicit update can be solved, in principle, since $H$ is known analytically.
%
% The procedure in Equation~\eqref{eq:implementation_b} appears  to be the most practical implementation of  implicit stochastic approximation, which justifies  why the implicit update of algorithm \eqref{eq:implementation_b} lends its name to implicit stochastic approximation.
One of the most popular applications of the procedure in Equation~\eqref{eq:implementation_b} is in iterative statistical estimation, where $H(\theta, \xi) = -\nabla \log \ell(Y; X, \theta)$, 
and $\ell$ corresponds to the likelihood of a random data point $\xi=(Y, X)$, at parameter value $\theta$.
For example, if in Equation~\eqref{eq:implementation_b} we use $H(\theta_{n-1}, \xi_n)$ instead of $H(\theta_n, \xi_n)$, this amounts to classical SGD, which 
is widely popular in optimization and signal processing~\citep{coraluppi1969stochastic}, and has been fundamental in modern machine learning 
with large data sets~\citep{amari1998natural, zhang2004solving, bottou2010, bottou2016optimization}.
% The theoretical properties of ISGD are fairly well understood. 
When we use the implicit update, as originally described in Equation~\eqref{eq:implementation_b}, 
then the resulting procedure is known as incremental proximal method in optimization~\citep{bertsekas2011incremental}, 
or as implicit stochastic gradient descent~(ISGD) in 
statistics and machine learning~\citep{toulis2014statistical}.
We refer readers to~\citep{bertsekas2011incremental} and \citep{toulis2017aos} for two complementary analyses of implicit SGD, including asymptotic and non-asymptotic errors; see also~\citep{bianchi2016ergodic, salim2019snake, bianchi:hal} for related analyses using monotone operator theory.

The substitution of $\theta_n^+$  with $\theta_n$ may naturally cause concerns about whether the stability properties of the proximal Robbins--Monro procedure carry over to the approximate implementation through ISGD.
All aforementioned related work generally points to the same fact: ISGD shows superior performance to classical SGD, both in theory and practice. In particular, ISGD has identical asymptotic efficiency 
and convergence rate as standard SGD, but it is significantly more stable numerically~\citep[Section 2.5]{toulis2014statistical}. 
% The numerical stability of ISGD is in fact inherited from the stability of  proximal Robbins-Monro as illustrated in Theorem~\ref{theorem:nonasymptotic}, where the initial conditions cannot be amplified due to misspecification of the learning rate.
 In contrast, in classical SGD the initial conditions can be amplified arbitrarily when 
the learning rate is misspecified, leading to numerical divergence~\citep[Theorem 1]{moulines2011non}.
To illustrate these stability advantages of ISGD we present two examples from the literature:
one example is on a linear normal model where the theoretical assumptions in Section~\ref{section:theory} of this paper hold,
and another example on a Poisson model where the assumptions do not hold because the objective is non-Lipschitz.

\subsubsection{Example: linear normal model}
\label{isgd:example_normal}
Let $\thetastar \in \Reals{p}$ 
be the true parameters of a normal model, 
$y | x \sim N( x^\top\thetastar, \sigma^2)$, where $x\in\Reals{p}$ and $y\in\Reals{}$.
Let $\xi = (y, x)$ denote one datapoint, and define $H(\theta, \xi) = -(y - x^\top \theta) x$ as above. 
Then, classical Robbins--Monro reduces to:
\begin{align}
\label{eq:sgd_normal}
\thetarm{n} & = (I-\gamma_n x_n x_n^\top) \thetarm{n-1}  + \gamma_n y_n x_n.
\end{align}
Procedure \eqref{eq:sgd_normal} is equivalent to classical SGD on the least-squares objective.
It is also known as the least mean squares filter (LMS) in signal processing, 
or as the Widrow-Hoff algorithm~\citep{widrow1960adaptive}.
From Equation~\eqref{eq:implementation_b}, the ISGD update for this problem can be solved in closed form:
\begin{align}
\label{eq:isgd_normal}
	\thetaim{n} & = \frac{1}{1+ \gamma_n \|x_n\|^2} \thetaim{n-1} + \frac{\gamma_n}{1+\gamma_n \|x_n\|^2} y_n x_n.
\end{align}
Procedure~\eqref{eq:isgd_normal} is also 
known as the normalized least mean squares filter (NLMS) in signal processing \citep{nagumo1967}.
From Equation~\eqref{eq:sgd_normal}, we see that it is crucial for classical SGD to have a well-specified learning rate parameter $\gamma_1$. For instance, assume fixed $\|x_n\|^2=c^2$, for simplicity, then if $\gamma_1 c^2 \gg 1$ 
the iterate $\thetarm{n}$ of classical SGD will diverge to a value 
$\bigO{2^{\gamma_1 c^2}/\sqrt{\gamma_1 c^2}}$~\citep[for example]{toulis2014statistical}.
In contrast, a very large $\gamma_1$ will not cause divergence in 
ISGD, but it will simply put more weight 
on the $n$-th observation, $y_n x_n$, as can be seen in Equation~\eqref{eq:isgd_normal}. 
Intuitively, from a statistical perspective, ISGD specifies an averaging of old and new information, by weighing the estimate and observation according to the inverse of information, 
$(1+\gamma_n \|x_n\|^2)$.

The stability advantages of ISGD on classical SGD in the normal model are further illustrated in  the numerical simulations of Section~\ref{sec:simulations}.

\subsubsection{Example: Poisson regression}
\label{isgd:example_poisson}

Following the setup in Section~\ref{isgd:example_normal}, 
now let $y | x \sim \mathrm{Pois}(e^{x^\top\theta_\star})$, where \text{``Pois''} denotes the Poisson density.
Then, the classical SGD procedure reduces to:
\begin{align}\label{eq:sgd_poisson}
\theta_n = \theta_{n-1} - \gamma_n (y_n - e^{x_n^\top\theta_{n-1}}) x_n.
\end{align}
The implicit SGD procedure for this problem is equivalent to:
\begin{align}\label{eq:isgd_poisson}
\theta_n = \theta_{n-1} - \gamma_n (y_n - e^{x_n^\top\theta_{n}}) x_n.
\end{align}
The implicit update can be easily implemented through the update
$\theta_n = \theta_{n-1} - \lambda x_n$, 
where $\lambda = f_n(\lambda)$ is the fixed point of $f_n(s) = \gamma_n (y_n - e^{x_n^\top\theta_{n-1} + s \|x_n\|^2})$.
Since $f_n$ is non-increasing, the search bounds for its fixed point
are to be found in $[0, f_n(0)]$ or $[f_n(0), 0]$ depending on whether $f_n(0) > 0$ or $f_n(0) < 0$, respectively.

Regarding stability, we can see that the updates in Equation~\eqref{eq:sgd_poisson} are extremely sensitive to specification of $\gamma_n$ due to the non-Lipschitzness of the objective in the Poisson model.
Implicit SGD, on the other hand, is more stable than classical SGD thanks to the implicit update. 
For example, when we start at $\theta_0=0$ and $\|x_1\|=O(1)$, the next 
iterate, $\theta_1$, will be $O(e^{\gamma_1 y_1})$ in classical SGD, which diverges arbitrarily as $\gamma_1$ increases.
On the other hand, implicit SGD has a very different behavior thanks to the implicit update in~\eqref{eq:isgd_poisson}:
when $\gamma_1$ is small such that $\gamma_1 y_1 \ll 1$, then $\theta_1$ is $O(\gamma_1 y_1)$; 
but when $\gamma_1$ is large, then $\theta_1$ asymptotes to $O(\log y_1)$.

These stability advantages of ISGD over classical SGD in the Poisson model are further illustrated in  the numerical experiments of Appendix~\ref{appendix:poisson}.
% Section~\ref{sec:simulations}.

\subsection{Approximate implementation with proximal stochastic fixed points}
\label{section:impl_nested}

In this section, we consider cases where we cannot observe  directly the random component $\xi$ 
of $H(\theta, \xi)$.
As mentioned earlier, this includes cases where the analytic form of $h(\theta)$ or $H(\theta, \xi)$ is unknown, and may only be sampled through, say, a sequential experiment.
% As already discussed in Section~\ref{section:introduction} and Section~\ref{section:related}, the procedure described in Equation~\eqref{eq:proxRM_fp} cannot be directly applied to this setting, since the intermediate value $\thetaMed{n}$ cannot be computedwithout knowledge of the regression function, $h(\theta) = \Ex{\Y{\theta}}$. 
%
We thus present an approximate implementation of the proximal Robbins--Monro 
procedure based on nested stochastic approximations that can be used without any auxiliary knowledge of the
estimation problem.  The nested procedure is in fact a proximal form of a fixed-point stochastic
approximation procedure~\citep[Section 10.2]{borkar2008stochastic}, which, however, is run
only for a finite number of steps. 
% To the best of our knowledge there is noanalysis of such procedures in the literature, so our convergence analysis in Theorem~\ref{thm:nested} applies novel techniques which may be of general interest.
%
Section~\ref{section:example} illustrates the benefits of the nested procedure
in quantile estimation. 

To begin, we first take expectations in the proximal Robbins--Monro iteration:
$$
\Ex\big(\theta_n - \theta_{n-1} - \gamma_n H(\theta_n^+, \xi_n)\big) = 0
\Rightarrow 
\Ex\big(\theta_n^+ - \theta_{n-1} - \gamma_n H(\theta_n^+, \xi_n)\big) = 0.
$$
The key idea is then to treat $\theta_n^+$ as the solution 
to $\Ex_\xi\big(\theta - \theta_{n-1} - \gamma_n H(\theta, \xi)\big) = 0$, 
and solve this characteristic equation through a separate, standard
stochastic approximation procedure. At every $n$-th iteration, we therefore
run a Robbins--Monro procedure, $w_k$, for $K$ steps  as follows:
\begin{equation}
	\label{eq:sfp}
	\begin{split}
		w_1 &= \thetaim{n-1},\\
		w_k &= w_{k-1} - a_k \big(\gamma_n H(w_{k-1}, \xi_k) + w_{k-1}
		- w_1\big),\quad 1 < k\leq K,\\
		\thetaim{n} &= w_k.
	\end{split}
\end{equation}

At first, it may seem that this procedure is affected by the same stability issues as classical stochastic approximation. However, our convergence result that follows will show that this is not true. For intuition, 
note that for fixed $n$ the sequence $(w_k)_{k\geq 1}$ is a standard Robbins--Monro
procedure applied to a different minimization problem:
\begin{equation}
	\label{eq:min}
	\min_{\theta\in\Theta} \left\{ \frac{1}{2\gamma_n} \|\theta-\theta_{n-1}\|^2 + \P(\theta)\right\}.
\end{equation}
What we gain compared to applying the classical Robbins--Monro method to $h$ directly, is that the objective function in Equation~\eqref{eq:min} is now strongly convex, even when $\P$ is not.
With this formulation, it is easy to verify that $\thetaMed{n}$ is the solution to this optimization problem, so that $w_k\to\theta_n^+$. 
Therefore, the problem structure that we designed allows the application 
of explicit updates, without compromising numerical stability.
We illustrate this point in Section~\ref{section:example}.

%To proceed with an analysis, we make use of strong convexity in Assumption~\ref{assumption:convexity}(\ref{A:h_inward2}).
%This allows an exponential discounting of the initial conditions, even when $\delta$ is very small, 
%as shown in the following theorem.
%
% \LemmaChi
%
%
% \LemmaXi
%
%
% \LemmaIdealized
%
%
\TheoremNested

\vspace{2mm}
Theorem~\ref{thm:nested} shows two key results.  First, the initial conditions
of the nested procedure are forgotten exponentially fast at a rate which can be
made arbitrarily close to $(1+\gamma\mu)^{-n}$, which was also true in the
idealized procedure. Second, an approximation error smaller than $\epsilon$ can
be obtained by choosing $n=O\big(\log\frac{1}{\epsilon}\big)$, and
$K=O\big(\frac{1}{\epsilon^2}\big)$, where $K$ is the number of iterations in
the inner procedure.  Taken together, these choices imply a total number of
gradient observations of order
$O\big(\frac{1}{\epsilon^2}\log\frac{1}{\epsilon}\big)$. 
In comparison, under
the same assumptions, the stochastic proximal point algorithm
(Section~\ref{section:theory}) and the standard Robbins--Monro procedure
achieve an approximation error smaller than $\epsilon$ using
$O\big(\frac{1}{\eps^2}\big)$ observations. Hence, the approximate
implementation of Equation~\eqref{eq:sfp} incurs a small (logarithmic)
overhead in terms of number of observations required to achieve a given level of
accuracy. Experiments in Section~\ref{section:example} and
Appendix~\ref{appendix:poisson}, however, show that this overhead may be
negligible in practice and that the stability benefit of procedure \eqref{eq:sfp} is
preserved without sacrificing accuracy, even when restricted to run for the
same amount of time as other methods.

%
%
%\remark{The proposed proximal stochastic fixed-point procedure and its theoretical analysis, constitute a key contribution of this paper.  We are unaware of other (approximate) proximal methods that address settings where the objective is analytically unknown, and where the underlying procedure is comprised of nested stochastic fixed points.
%}

\remark{The proof of Theorem~\ref{thm:nested} is technically challenging due to the nested nature of the
procedure. 
This requires careful balancing of the accumulation of approximation
errors from the inner iteration jointly with the rate of convergence of the
idealized procedure.  
To the best of our knowledge, there are are no such nonasymptotic analyses of stochastic fixed-point procedures in the literature.
The proof of Theorem~\ref{thm:nested} therefore applies novel techniques, which may be of general interest.
We also note that convergence of the nested procedure when $\P$ is non-strongly convex is an open question, which we leave 
for future work.
}

\remark{The nested nature of the procedure described in Equation~\eqref{eq:sfp} is reminiscent of the Catalyst scheme of~\citet{lin2015universal}, which is a general acceleration technique for first-order optimization methods. Similar to the Catalyst scheme, our procedure \eqref{eq:sfp} approximately computes a proximal update at each iteration. The key difference is that we analyze how to perform this approximate computation, whereas the Catalyst scheme assumes oracle access to such computation. Furthermore, the main focus of the Catalyst scheme is to achieve acceleration à-la-Nesterov with the use of a momentum term, while our focus is to analyze the stability of proximal updates.}

In the following sections, we illustrate the use of the 
nested procedure of Equation~\eqref{eq:sfp} and the use of
Theorem~\ref{thm:nested} through a simulated study on the normal model of Section~\ref{isgd:example_normal}, and the classical quantile
estimation problem of~\citet{robbins1951}.

\section{Simulated studies on stability}\label{sec:simulations}

Here, we investigate empirically the stability of plug-in implementations of the
proximal Robbins--Monro procedure presented in Section~\ref{section:impl_plugin}. Specifically, we present results for the normal linear model 
of Section~\ref{isgd:example_normal}. In Appendix~\ref{appendix:poisson}, we present results for the Poisson model of Section~\ref{isgd:example_poisson}, which are 
even more favorable towards the proximal Robbins--Monro than the normal model.

Our simulation setting is as follows. We consider parameters $\theta_\star\in\Reals{p}$, $p=6$, such that 
$\theta_{\star, j} = e^{-j} (-1)^j$, for $j=1, \ldots, 6$; and $y|x \sim N(x^\top\theta_\star, \sigma^2)$, 
where $x\sim N_p(0, \Sigma)$ is a $p$-variate normal with $\Sigma = 2 I + u u^\top$, with $u$ a column vector with $p$ iid uniform random variables, $U(0, 1)$; we set $\sigma^2=4$.
We estimate recursively $\theta_\star$ using the procedures of SGD and ISGD as introduced in 
Equations~\eqref{eq:sgd_normal} and~\eqref{eq:isgd_normal}, respectively. 
We also use the stochastic fixed point method of Equation~\eqref{eq:sfp} to estimate $\theta_\star$. To satisfy the conditions of Theorem~\ref{thm:nested} we set:
\begin{align}\label{eq:specs_sfp_normal}
a = \log(L/\mu),~K = 2a(1+\gamma_1 L)^2,~a_k = 2a/K.
\end{align}
The parameters $L, \mu$ are estimated directly from data~(in the linear model, these values 
correspond to the maximum and minimum eigenvalue of $\Sigma$, respectively).~Finally, the learning rate for all procedures is set as $\gamma_n = \gamma_1/n$ across iterations, and 
we vary $\gamma_1$ in the experiment to check the sensitivity of the procedures to specification of the learning rate.

As we vary $\gamma_1$, we replicate datasets of size $N=10,000$ based on the above model setup, 
and run the three procedures above for a total wall clock time of 1 second. For every replication, 
we calculate the trajectory of log mean squared error~(log-MSE) of all procedures, $m_{j,n} = \|\theta_{j, n} - \theta_\star\|^2$, where $\theta_{j, n}$ is the $n$-th iterate of procedure $j$ within the data replication.
From this series, we are interested in two summary statistics. First, the ``mean-level" of the log-MSE, $m_{j, n}$, 
which corresponds to the level around which the log-MSE ``settles''. To calculate this number, 
we fit an AR(1) model on the series $m_{j, n}$, and then 
calculate the stationary limit $b_1/(1-b_0)$, where $b_1$ is the estimated slope coefficient and $b_0$ is the estimated intercept in the model. 
Second, we are interested in the maximum value, $\max_{i : t_{j,i} \le 1} \{m_{j, i}\}$, of log-MSE across iterations, where $t_{j,i}$ is the wall clock time until iteration $i$ for method $j$. This acts as a proxy for the sensitivity of the procedure.

\begin{figure}[t!]
\centering
	\label{figure}
	\includegraphics[width=1\textwidth]{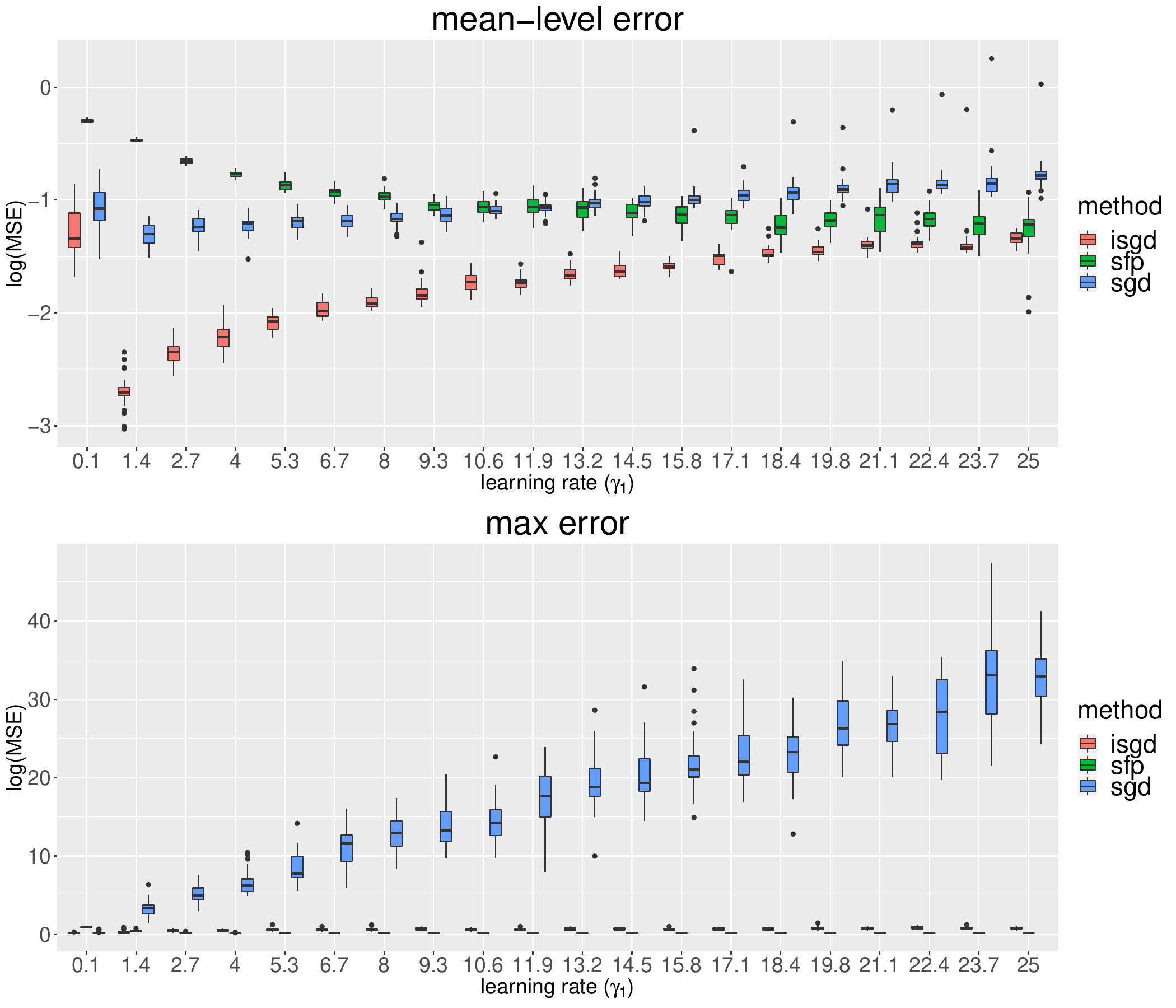} 
	\caption{\textbf{Top:} Boxplots of mean-level log-MSE over 50 replications of (i) the classical Robbins--Monro
procedure~(``sgd")~of Equation~\eqref{eq:sgd_normal}; 
(ii) the nested implicit stochastic approximation procedure~(``sfp") of Equation~\eqref{eq:sfp};
and (iii) the implicit SDG procedure~(``isgd") of Equation~\eqref{eq:isgd_normal}.
{\bf Bottom:} Boxplots of maximum lo- MSE over all iterations for each method.
Each procedure runs for a total of 1 sed. of wall-clock time.
	}
\label{figure:normal}
\end{figure}

Figure~\ref{figure:normal} shows the results of this experiment. For any value of $\gamma_1$, 
Figure 1(top) shows the boxplot of the mean-level log-MSE for each procedure. We see that each  procedure
behaves differently. Across all $\gamma_1$ values, ISGD performs best, and also remains robust. The stochastic fixed-point procedure (SFP) starts from worst performance for small $\gamma_1$. However, as $\gamma_1$ increases 
SFP keeps improving in MSE, and its variance increases as well. This can be explained by the SFP specification in Equation~\eqref{eq:specs_sfp_normal}, where larger $\gamma_1$ lead to larger $K$. Since 
we keep the computation budget (measured in wall-clock time) fixed, this means that as $\gamma_1$ increases 
SFP performs more inner iterations (large $K$) but fewer outer iterations (small $n$).
In contrast, classical SGD is the most unstable procedure. We see that its MSE steadily increases, while the maximum MSE (Figure~\ref{figure:normal}, bottom) varies widely, as $\gamma_1$ increases.

This experiment illustrates a key point of our paper: 
classical Robbins--Monro methods are sensitive to parameter specifications, 
while proximal Robbins--Monro methods, and even approximate implementations of it, remain stable in a wide range of specifications.

\section{Application: Iterative quantile estimation}
\label{section:example}
%%%
In their seminal paper,~\citet{robbins1951} applied stochastic approximations in iterative quantile estimation. In this problem, 
$H(\theta, \xi)$ corresponds to a sample drawn from a distribution with cumulative distribution function $Q(\theta)$. The goal 
is to estimate $\theta_\star$
such that $Q(\theta_\star) = \alpha$, for given $\alpha\in(0, 1)$.
A relevant application from medicine and toxicology is the estimation of the dose that is lethal to 50\% of experimental subjects, known as LD50~\citep{grieve1996likelihood}.

% In particular, consider a random variable $Z$ with a cumulative distribution function $F$. An experimenter wants to find the point $\thetastar$ such that  $F(\thetastar)=\alpha$, for a fixed $\alpha \in (0, 1)$.  The experimenter cannot draw samples of $Z$, but has access to the random variable $W_{\theta} = \mathbb{I}\{Z \le \theta\} - \alpha$, for any value of $\theta$.

In more detail, consider a random variable $\xi$ with cumulative distribution function $Q$. An experimenter wants to find the point $\thetastar$ for which  $Q(\thetastar)=\alpha$, for some fixed $\alpha \in (0, 1)$.  Let $h(\theta) = Q(\theta)-\alpha$. 
The experimenter cannot observe $\xi$ directly, but has only access to 
$\mathbb{I}\{\xi \le \theta\}$, for any value of $\theta$.
\citet{robbins1951} showed that the following iterative procedure,
\begin{align}
\label{eq:rm_quantile}
\theta_{n} = \theta_{n-1}-\gamma_n H(\theta, \xi_n),
\end{align}
where $H(\theta, \xi) =  \mathbb{I}\{\xi \le \theta\} - \alpha$, converges to $\theta_\infty$ for which $\Ex(H(\theta_\infty, \xi))=0$.
Consequently, it solves $\Ex(\mathbb{I}\{\xi\le\theta_\infty\})-
\alpha = Q(\theta_\infty)-\alpha = 0$, and by monotonicity of $Q$, we obtain $\theta_\infty = \theta_\star$.

Despite such theoretical convergence, the numerical stability of the Robbins--Monro procedure can be
 challenged by the following result.
\begin{proposition}
\label{prop1}
Assume that $\theta_0<\theta_\star$  and that
	$\theta_0+\gamma_1\alpha >\theta_\star$, then for any $\eps>0$ such that
	$\theta_0+\gamma_1\alpha>\theta_\star+\epsilon$, with probability
	$1-Q(\theta_0)$, the number of iterations $N_\epsilon$ of procedure~\eqref{eq:rm_quantile} required to
	approximate $\theta_\star$ within accuracy $\epsilon$ is lower-bounded:
	\begin{equation}
		\label{eq:lb}
		\log N_\epsilon \geq
		\frac{\theta_0+\gamma_1\alpha -\theta_\star-\epsilon}{(1-\alpha)\gamma_1}
		\;.
	\end{equation}
\end{proposition}
\begin{proof}
	With probability $1-Q(\theta_0)$ the first iterate of \eqref{eq:rm_quantile} is
	$\theta_1 = \theta_0 + \gamma_1\alpha>\theta_\star$, where the inequality
	is by assumption. Conditioned on this event, the progress in each
	subsequent iteration, namely $\theta_{n}-\theta_{n-1}$, is upper-bounded by
	$\gamma_n (1-\alpha)$ with probability 1 as long as
	$\theta_n >\theta_\star$. This implies that
$
		\theta_n \geq \theta_0 +\gamma_1\alpha - (1-\alpha)\sum_{k=2}^n\frac{\gamma_1}{k}
		\geq \theta_0 +\gamma_1\alpha - (1-\alpha)\gamma_1\log n$.\qed
\end{proof}

Proposition~\ref{prop1} essentially shows that there are values of the learning rate
parameter $\gamma_1$ and initial estimate $\theta_0$ for which the classical Robbins--Monro procedure may be stuck indefinitely. 
For example, let $Q$ be the standard normal distribution, and let
$\alpha=0.999$, so that $\theta_\star = 3.09$ is the solution.  Suppose also
that $\gamma_1 = Q'(\theta_\star)^{-1}\simeq 297$, which is the learning rate value suggested
by standard theory~\citep{nemirovski2009robust}.  Let $\theta_0 = -10$ and
suppose that $H(\theta_0, \xi_1) = -\alpha$. It
follows that 
$$
\theta_1 = -10 - \gamma_1 (-\alpha) = -10 + \gamma_1\alpha 
\approx 287 \gg \theta_\star.
$$
From there, the Robbins--Monro procedure makes progress by at most $\gamma_i
(1-\alpha)\simeq \frac{297}{i}\cdot 10^{-3}$ at each step.
Thus, the number of
iterations required to return back from $\theta_1$ to a region near $\theta_\star$ 
is at the order of $e^{956}$. In other words,
the procedure gets stuck at large values of
$\theta$, where the derivative of the objective is negligible. 
% The case where $\theta_1>\theta_\star$ follows from a symmetric argument.

This numerical example illustrates that a misspecification of
$\gamma_1$ can dramatically amplify the initial conditions in classical stochastic approximation, 
and affect convergence.
It is therefore interesting to investigate whether the proximal Robbins--Monro method
offers an improvement.

\subsection{Stability of the proximal stochastic fixed points}
\label{sec:quantiles}
In the context of quantile estimation, the stochastic fixed point procedure of Equation~\eqref{eq:sfp} 
can be written as follows:
\begin{align}
\label{eq:sfp_quantile}
w_1 &= \thetaim{n-1},\nonumber\\
		w_k &= w_{k-1} - a_k \big(\gamma_n H(w_{k-1}, \xi_k) + w_k
		- w_1\big),\quad 1< k\leq K, \nonumber\\
		\thetaim{n} &= w_k,
\end{align}
where $H(\theta, \xi) = \mathbb{I}\{\xi \le \theta\}-\alpha$, $\gamma_n=\gamma_1, a_k=2a/K$, and $\gamma_1, a$ and $K$ are constants to be defined. 

Before presenting our numerical experiments, 
we discuss intuitively why the nested procedure in
Equation~\eqref{eq:sfp_quantile} improves upon the classical Robbins--Monro method in Equation~\eqref{eq:rm_quantile}, 
and also discuss how to define the constants according to Theorem~\ref{thm:nested}.
We address these two issues successively.
First, consider the idealized case where $K=\infty$. 
In this case, the  iteration in Equation~\eqref{eq:sfp_quantile} converges to the solution of the following 
fixed-point equation:
\begin{displaymath}
w_\infty = \theta_{n-1}-\gamma_n \big(Q(w_\infty)-\alpha\big).
\end{displaymath}
The next iterate, $\theta_n$, is simply defined as $\theta_n=w_\infty$. 
It is easy to verify the stability of this fixed point. 
For example, if $\theta_{n-1}< \theta_\star$, then
$\theta_{n-1}<\theta_n<\theta_\star$;
and, conversely, if
$\theta_{n-1} > \theta_\star$, then $\theta_\star<\theta_n <\theta_{n-1}$. That
is, the idealized procedure with $K=\infty$ always
pulls back in the right direction towards $\theta_\star$, 
and thus always makes progress towards the global solution. Convergence is also extremely
fast, as shown in the proof of Theorem~\ref{thm:nested}. To illustrate
numerically, consider the example of the previous section where the classical
Robbins--Monro procedure did not converge. Using the same 
numbers, at the second iteration the idealized procedure will calculate:
$$
\theta_1 = -10 - 297 \big(Q(\theta_1) - .999\big),
$$
which solves to $\theta_1 \approx 1.74$; if we keep iterating, the idealized procedure will be 0.01-close to $\theta_\star$ by 
the hundredth iteration. This is a vast improvement compared to
the classical Robbins--Monro method, which remains stuck virtually forever.

Second, consider the actual nested procedure in
Equation~\eqref{eq:sfp_quantile}, where $K$ is finite. Theorem~\ref{thm:nested}
shows that the procedure maintains the nice convergence and stability
properties of the original procedure under certain assumptions.  The
assumptions in this case can be greatly simplified if we consider that for the
normal distribution, the probability density function is upper-bounded. Hence, $L\leq 1$ and Theorem~\ref{thm:nested} suggests the following choice of hyperparameters for the nested procedure:
\begin{align}
\label{eq:a1}
\gamma_n = \gamma_1, a = \frac{1}{(1 + \gamma_1)^2}, 
\text{ and } K=50.
\end{align}
Note in particular that this choice of parameters satisfies $K \geq 2a
(1+\gamma_1 L)^2$, as required by the theorem.  We can define the constants in a similar manner for
arbitrary distributions from an upper bound on the probability density
function.
Next, we evaluate numerically the (approximate) proximal Robbins--Monro procedure resulting 
from the aforementioned choice of hyperparameters.

% In the following section we give an extensive experimental comparison between implicit stochastic approximation and the classical Robbins-Monro method.  

\subsection{Numerical evaluation}

Here, we conduct a numerical evaluation of our proposed nested procedure~in
Equation~\eqref{eq:sfp_quantile}, using the parameter settings of
Equation~\eqref{eq:a1}, and compare it with the classical Robbins--Monro procedure in Equation~\eqref{eq:rm_quantile}. For a fair comparison, both methods run for a total of 1 second of wall-clock time. 
As the iteration complexity is similar for both methods, 
the classical Robbins--Monro procedure runs for $N$ iterations, whereas the stochastic fixed point 
runs roughly for $N/K$ outer iterations, and $K$ inner iterations.
This way, the total number of random samples (gradient
observations) used by our procedure similar to those in the classical
procedure.

As mentioned before, $Q(\theta)$ is here the cumulative distribution function of the standard normal,
$\alpha=0.999$ and $\theta_0=-10$. The quantity to be estimated is
$\theta_\star\approx 3.09$, for which $Q(\theta_\star) = \alpha$. 
For different values of $\gamma_1$ we compare the
Robbins--Monro procedure to our proposed fixed point procedure in Equation~\eqref{eq:sfp_quantile}, with $K=50$. 
% As also explained before,  both methods run as many iterations as possible for a total of 1 second of wall-clock time.
%
For each value of $\gamma_1$, the experiment is
replicated 100 times, and we report an average of all 
final estimates from both procedures.
% $\theta_N$ for Robbins--Monro, and $\theta_{N/K}$ for the nested procedure.
The results of this experiment are shown in Figure~\ref{figure:RM}.

\begin{figure}[t!]
\centering
\includegraphics[width=\textwidth]{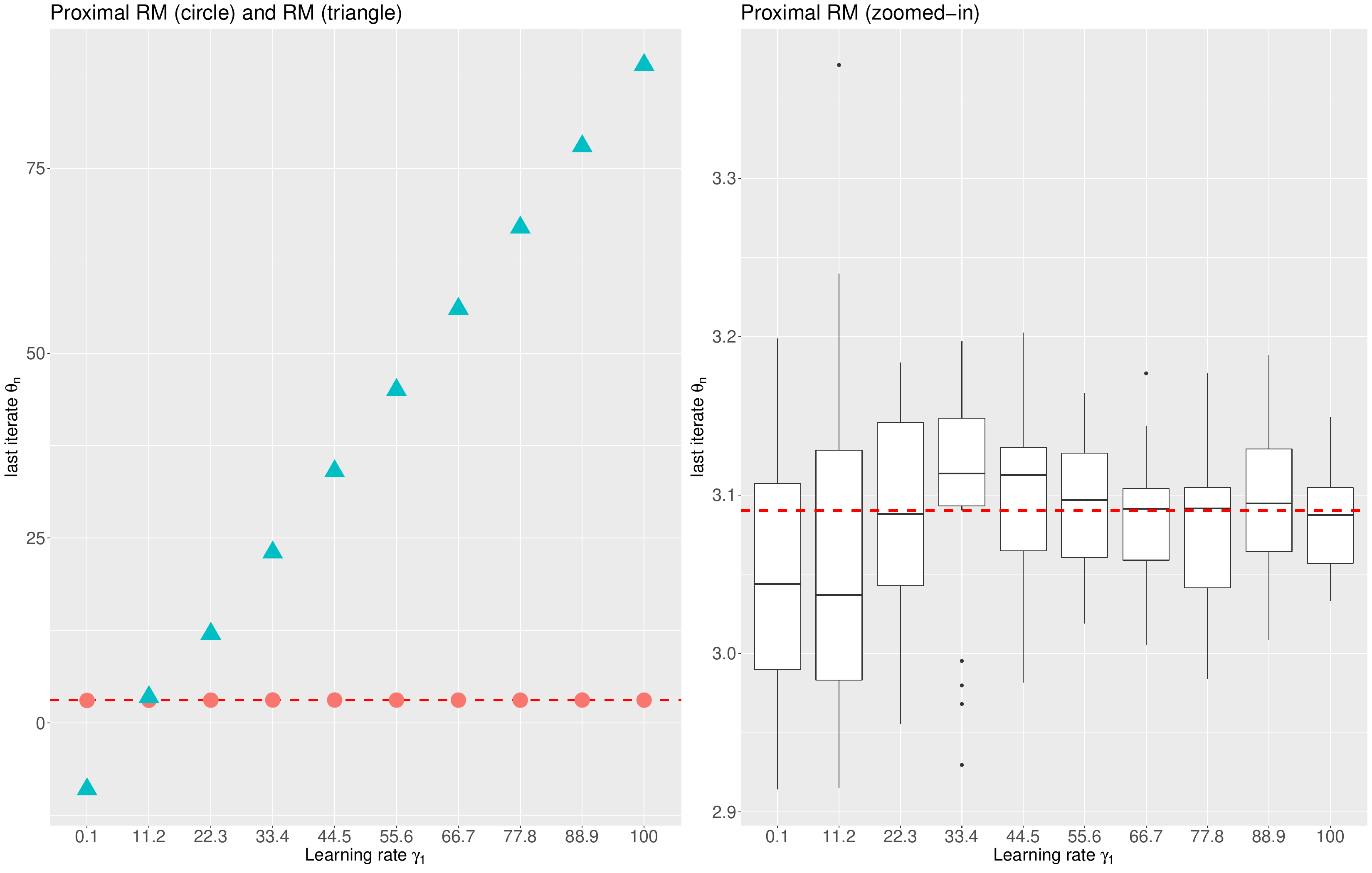}
	\caption{\label{figure:RM}\textbf{Left:} boxplots of 100 replications of the Robbins--Monro (RM)
procedure of Equation~\eqref{eq:rm_quantile} and of Proximal Robbins--Monro(Prox-RM), approximately 
implemented by~\eqref{eq:sfp_quantile}; averages of last iterates for RM and Prox-RM are indicated as triangles and circles, respectively. Each procedure runs  for a total of 1 second of wall-clock time.
	\textbf{Right:} Zoom in to proximal RM (note
	the different scale on the $y$-axis). Left plot is in log-scale.
	% negative	values (for $\gamma_1=0.1$ and $\gamma_1=.5$) are not shown for RM.
%
The dashed horizontal line depicts true value, $\thetastar= 3.09$.
Both procedures start from $\theta_1=-10$, 
	and prox-RM is implemented following Equation~\eqref{eq:a1}.
We see that prox-RM is more stable to specification of $\gamma_1$ than classical RM.
% increases, the classical
%RM method overshoots and essentially remains stuck, 
%	which explains the flat boxplots.  In contrast, Prox-RM remains robust, with final
%	iterates estimating the true value well, except for a small bias 
%	at very small or very large values of the learning rate.
}
\label{figure:quantiles}
\end{figure}

In the left plot, we observe that the classical Robbins--Monro procedure indeed suffers from numerical
instability.  In particular, as predicted by
Proposition~\ref{prop1}, when $\gamma_1$ increases beyond
$\frac{\thetastar-\theta_0}{\alpha}\simeq 13.1$, the iterates overshoot and
remain virtually stuck for all subsequent iterations. 
%This explains why the
%boxplots for the Robbins--Monro method look flat for large values of
%$\gamma_1$; for small values of $\gamma_1$ the iterates
%also do not vary much because their variance depends on $\gamma_1$.
%
In fact, there is only a small range of values for $\gamma_1$ (visually 
between values 11 and 15), for which $\gamma_1$ is big enough to allow convergence, yet
small enough to prevent the aforementioned numerical instability.  Not shown in
the figure, the estimates from the  Robbins--Monro procedure are negative for very small learning
rates; for example, when $\gamma_1=0.1$ the average estimate is $-8.8$. This is
close to the starting point, $\theta_1=-10$, and indicates that the classical
procedure makes little progress  when the learning rate is
very small.  This shows that classical Robbins--Monro approximations are extremely
sensitive to  specification of the learning rate values.

The results for proximal Robbins--Monro, as approximately implemented by the fixed point procedure of Equation~\eqref{eq:sfp_quantile}, are drastically different.  In the left subplot of
Figure~\ref{figure:quantiles}, we see that proximal Robbins--Monro neither overshoots nor undershoots in contrast to the classical
procedure.  We see that the proximal procedure maintains a remarkable numerical stability
across the entire range of learning rate values. The procedure is also
statistically efficient in that the final iterates are
centered around the true value (red dashed line) with small variance around it.
This is better shown in the right subplot of Figure~\ref{figure:quantiles},
which only focuses on the estimates of the proximal Robbins--Monro. 
We note that a slight bias exists for very small or
very large values of the learning rate. 
For example, the average parameter estimate is roughly $3.04$ when
$\gamma_1=0.1$. The bias goes
away, however, with increased sample sizes. 
% Notably, the bias is stable even for large values of $\gamma_1$. 

We emphasize again that, similar to the simulation studies of Section~\ref{sec:simulations}, the stochastic fixed point procedure is implemented
in a fully data-driven way, by choosing its parameters using
Equation~\eqref{eq:a1}, as prescribed by Theorem~\ref{thm:nested}.

\section{Concluding remarks}
The theoretical and empirical results presented in this paper 
point to key advantages of the proposed proximal Robbins--Monro procedure,
as defined in Equation~\eqref{eq:proxRM}, over the classical procedure of Robbins--Monro.
One such advantage is numerical stability.
Our theoretical analysis showed that such stability is obtained without sacrificing convergence or efficiency.
However, the proposed  method is idealized because it can only be approximately implemented.

While in this paper we propose two approximate implementations that work well in general settings, there remain several open questions.
First, although the implicit stochastic gradient methods described in Equation~\eqref{eq:implementation_b} are easy to 
implement in a wide class of models (e.g., generalized linear models, M-estimation), 
their application to large-scale  non-convex settings, such as neural networks, has just started to emerge~\citep{fagan2018robust}. In this context, the stability of proximal Robbins--Monro approximations appears to be beneficial 
as predicted by the theory in this paper. More work needs to be done, however, to analyze these settings theoretically, and to leverage the added flexibility in designing the learning rate sequence.

Second, extending the scope of nested, fixed point implementations of proximal Robbins--Monro as in Equation~\eqref{eq:sfp}, is interesting especially because the procedure can operate even when only
samples from the objective are available. This introduces minimal modeling assumptions, which may be desirable in many settings, such as in econometric models, or in sequential experimentation of clinical trials. It is also an open question 
whether the substantive results of the quantile estimation example of Robbins--Monro presented in Section~\ref{sec:quantiles} extend to broader applications and domains.
We provided positive empirical evidence  in the simulations of Section~\ref{sec:simulations} and~Appendix~\ref{appendix:poisson},  and conjecture that this holds true more generally.

% In conclusion, we believe that the new stochastic approximation framework presented in this paper can provide a template for novel procedures in iterative estimation and machine learning that are numerically stable and statistically efficient, including parametric and non-parametric approaches.

\section{Acknowledgements}
This work was supported, in part, by the National Science Foundation under grants CAREER IIS-1149662 and IIS-1409177, by the Office of Naval Research under grants YIP N00014-14-1-0485 and N00014-17-1-2131, and by a Shutzer Fellowship to EMA.
Panos Toulis is grateful for the John E. Jeuck Faculty Fellowship at Booth.
The authors wish to thank Leon Bottou, Francis Bach, Adil Salim, Pascal Bianchi, Walid Hachem, and 
participants at the NESS conference for valuable comments and feedback.

\small
\bibliographystyle{Chicago}
\bibliography{refs}

\newpage

\setcounter{assumption}{0}
\setcounter{theorem}{0}
\setcounter{lemma}{0}

\appendix 

\section{Proofs of theorems for main method}
\label{appendix:main}
  \StateAssumptions\
\noindent {\em Note about proofs.} 
A key equation of implicit stochastic approximation is Equation~\eqref{eq:spp0}:
\begin{align}
\label{eq:FP}
\theta_n^+  + \gamma_n h(\theta_n^+) = \theta_{n-1}.
\end{align}
As this fixed-point equation has a unique solution, $\theta_n^+$ is a {\em deterministic
function} of $\theta_{n-1}$.
% By assumption, $H(\thetaMed{n}, \xi_n) = h(\theta_n^+) + \varepsilon_n$,  and so $\Ex\{H(\thetaMed{n}, \xi_n) | \mathcal{F}_{n-1}\} = h(\theta_n^+)$.
    %
  \TheoremConvergenceProof
  \TheoremDevianceProof
  \TheoremMSEProof
  \TheoremNormalityProof

  \section{Proofs for approximate implementations}
\label{appendix:nested}
\noindent {\em Note about proofs.} 
% We define the operators $\chi_n$ and $\zeta_n$ to analyze the nested procedure of Section~\ref{section:impl_nested}. In particular, $\zeta_n(\theta)$ will denote the output of procedure in Equation~\eqref{eq:implementation_a}, which is run for $K$  iterations (a fixed $K$ will be implicitly assumed). 
The procedures analyzed in this section involve two nested iterative processes.
Throughout, we use $n$ as the index variable of the outer iteration and $k$ for
the inner iteration. The randomness entering the $k$th step of the inner
iteration inside the $n$th step of the outer iteration is denoted by $\xi_k^n$
and $\Fn{n,k}$ denotes the $\sigma$-algebra generated by $\{\xi_i^j\}_{1\leq
i\leq K}^{1\leq j\leq n-1}\cup\{\xi_i^n\}_{1\leq i\leq k}$. We also write
$w_k^n$ instead $w_k$ in \eqref{eq:sfp} to explicitely keep track of the outer
iteration index.  Finally, we use $\Fn{n-1}$ as a shorthand for $\Fn{n-1,K}$.

Let $\chi_n(\theta)$ denote the output of the same procedure in the 
theoretical case where $K=\infty$. In other words, $\chi_n$ is the proximal operator 
that satisfies:
\begin{align}
\label{eq:chin}
\chi_n(\theta) + \gamma_n h(\chi_n(\theta)) = \theta.
\end{align}%
  \LemmaChiProof
  \LemmaXiProof

  \TheoremNestedProof

\section{Stability of Proximal Robbins--Monro: Poisson model}
\label{appendix:poisson}

In this section, we investigate empirically the stability of plug-in implementations of 
proximal Robbins--Monro presented for the Poisson model
of Section~\ref{isgd:example_normal}. This model has a non-Lipschitz likelihood, 
so the experiment is meant to test the performance of our theory and methods when our working assumptions do not hold.
We repeat a lot of information from the normal model experiment 
of Section~\ref{sec:simulations} for reader's convenience.

Our simulation setting is as follows. 
As before, we consider parameters $\theta_\star\in\Reals{p}$, $p=6$, such that 
$\theta_{\star, j} = e^{-j}$, for $j=1, \ldots, 6$; and $y|x \sim \mathrm{Pois}(e^{x'\theta_\star})$, 
where $x_{ij}$ takes values $\{0, 1, 2, 3\}$ with probabilities 
$\{0.4, 0.4, 0.15, 0.05\}$, respectively. 
We estimate recursively $\theta_\star$ using the methods of SGD and ISGD as introduced in 
Equations~\eqref{eq:sgd_normal} and~\eqref{eq:isgd_normal}, respectively. 
We also use the stochastic fixed point method of Equation~\eqref{eq:sfp} to estimate $\theta_\star$. As before we set:
$a = \log(L/\mu),~K = 2a(1+\gamma_1 L)^2,~a_k = 2a/K$.
The parameters $L, \mu$ are estimated directly from data. As expected, the estimated $L$ is 
larger in the Poisson model than in the normal model, and the estimated value increases with more observations.
This leads to larger values for $K$ and fewer outer iterations for the stochastic fixed point procedure.

The learning rates for all methods is set as $\gamma_n = \gamma_1/n$ across iterations.
As we vary $\gamma_1$, we replicate datasets of size $N=10,000$ based on the above model setup, 
and run the three procedures above for a total wall clock time of 1 second. For every replication, 
we calculate the trajectory of log mean squared error~(log MSE) of all methods, $m_{j,n} = \|\theta_{j, n} - \theta_\star\|^2$, where $\theta_{j, n}$ is the $n$-th iterate of method $j$ within the data replication.
The ``mean-level" of $m_{j, n}$ corresponds to the level around which the series ``settles''. As 
in the normal model, we calculate this number by fitting an AR(1) model on the series $m_{j, n}$, and then 
calculating the stationary limit $b_1/(1-b_0)$, where $b_1$ is the estimated slope coefficient and $b_0$ is the estimated intercept in the model. 
Second, we also calculate $\max_{i : t_{j,i} \le 1} \{m_{j, i}\}$, of log MSE across iterations, where $t_{j,i}$ is the wall clock time until iteration $i$ for method $j$. 
This max value is a proxy for the sensitivity of our method.

\begin{figure}[t!]
\centering
	\label{figure}
	\includegraphics[width=1\textwidth]{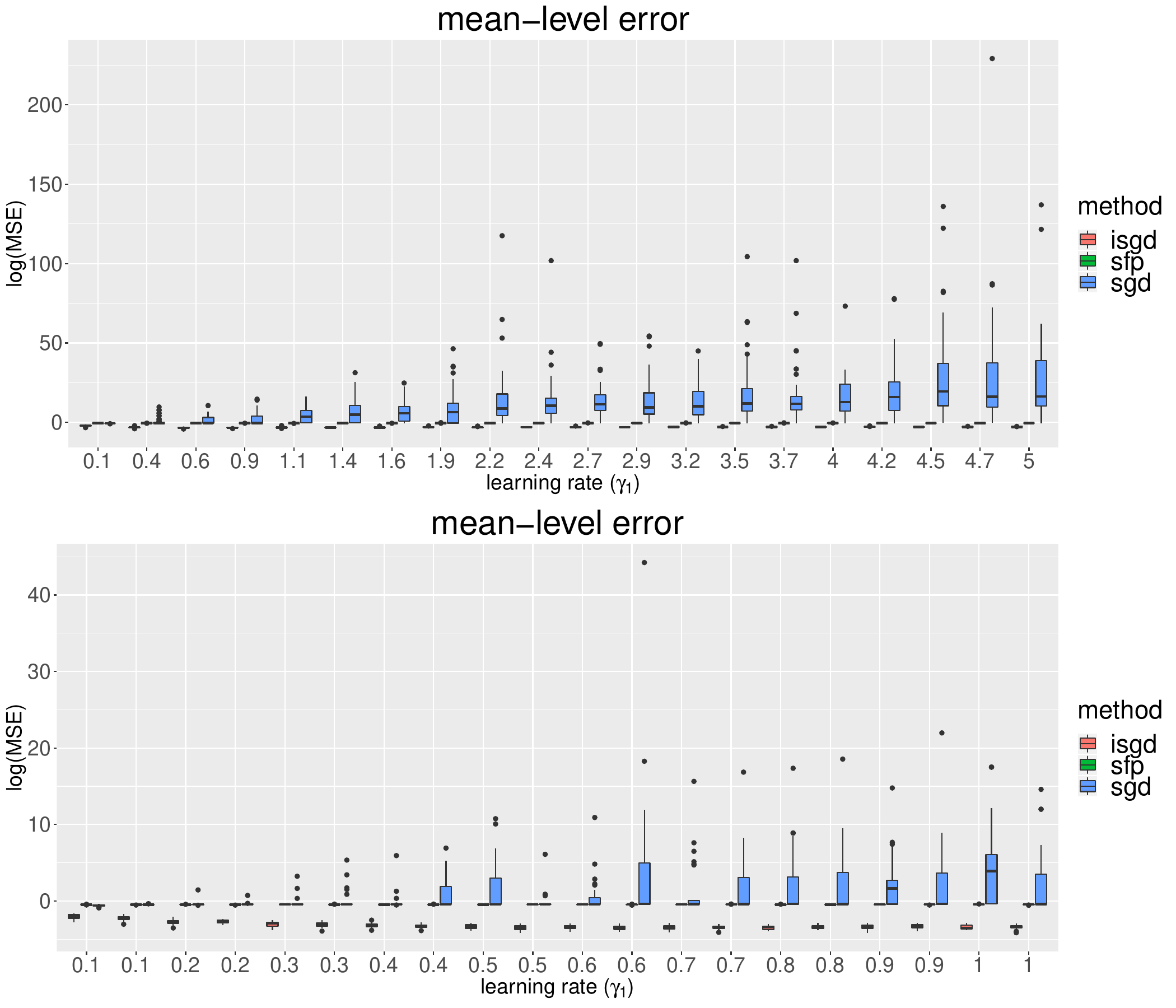} 
	\caption{\textbf{Top:} Boxplots of mean-level log MSE over 50 replications of SGD, ISGD, and 
	stochastic fixed point (SFP). 
{\bf Bottom:} Zoomed-in plot up to $\gamma_1=1$.
Each method is constrained to run for a total of one second of wall clock time.
	}
\label{figure:poisson}
\end{figure}

Figure~\ref{figure:poisson} shows the results of this experiment. 
For any value of $\gamma_1$, 
Figure~\ref{figure:poisson}(top) shows the boxplot of the mean-level log MSE for each method, 
while Figure~\ref{figure:poisson}(bottom) zooms in the mean-level plot for smaller values of $\gamma_1$ (up to .3) We see that both ISGD and SFP are clearly more stable than classical SGD. 
The latter essentially diverges even for small values of the learning rate.

This experiment suggests that the proximal Robbins--Monro methods, including the approximate implementations discussed in this paper, are generally more stable than their classical counterparts, when the theoretical assumptions do not hold.
Specifically, in this example, the likelihood is non-Lipschitz, and we see that classical SGD diverges even with slight misspecifications of the learning rate. While ISGD is known to be stable in generalized linear models~\citep{toulis2014statistical}, we believe it is remarkable 
that the SFP procedure, initialized with the theoretical values suggested by Theorem~\ref{thm:nested}, is stable in this highly non-linear model as well.

\section{Computation of implicit updates}\label{sec:computation}
At a first glance, the computation of the implicit procedure,
$$
\thetaim{n} = \thetaim{n-1} - \gamma_n H(\theta_n, \xi_n),
$$
may appear to be challenging, or even impossible. However, the implementation can actually be 
quite straightforward in a variety of popular models and objectives. The general 
idea is to exploit a special structure $W_\theta$ to simplify the implicit update. 

Specifically, suppose that $H(\theta, \xi) = s(\theta) U$, where $s(\theta)\in\mathbb{R}$ and $U$ is a vector that does not depend on the parameter value, $\theta$.
Then, we can write the implicit update as follows:
$$
\thetaim{n} = \thetaim{n-1} - \gamma_n s({\theta_n}) U_n = 
\thetaim{n-1} -\eta U_n,
$$
for some scalar $\eta$. Thus, we have to solve:
$$
\gamma_n s({\theta_n}) = \eta
\Leftrightarrow
\gamma_n s(\theta_{n-1}- \eta U_n) = \eta.
$$
The problem is now reduced to a one-dimensional fixed-point equation for $\xi$.
In many statistical models, including generalized linear models and M-estimation, 
this fixed point can be efficiently solved through line search due to the structure of $s$.
For instance, Algorithm 1 of \citet{toulis2014statistical} provides a concrete algorithm for  generalized linear models.

\end{document}